\documentclass[DIV=13]{scrartcl}
\usepackage{amsmath,amsthm,enumitem,upref}
\usepackage{mathtools}
\usepackage{microtype}
\usepackage[sortcites=true,giveninits=true,maxbibnames=10, backend=biber]{biblatex}
\AtEveryBibitem{\clearfield{issn}\clearfield{doi}\clearfield{urldate}}
\AtEveryCitekey{\clearfield{issn}\clearfield{doi}\clearfield{urldate}}
\renewrobustcmd*{\bibinitdelim}{\,}
\DeclareSourcemap{
  \maps[datatype=bibtex]{
    \map[overwrite]{
      \step[fieldsource=shortjournal,fieldtarget=journaltitle]
    }
  }  
}
\theoremstyle{definition}
\newtheorem{definition}{Definition}[section]
\newtheorem{example}[definition]{Example}
\newtheorem{remark}[definition]{Remark}
\newtheorem{corollary}[definition]{Corollary}
\theoremstyle{plain}
\newtheorem{theorem}[definition]{Theorem}
\newtheorem{lemma}[definition]{Lemma}
\newtheorem{proposition}[definition]{Proposition}
\newcommand{\C}{\mathbf{C}}
\newcommand{\R}{\mathbf{R}}
\newcommand{\N}{\mathbf{N}}

\newcommand{\Ll}{\mathrm{L}}
\newcommand{\Cc}{\mathrm{C}}
\newcommand{\LL}{\mathcal{L}}
\newcommand{\Ee}{\mathrm{E}}
\newcommand{\e}{\mathrm{e}}
\newcommand{\dd}{\,\mathrm{d}}
\newcommand{\Abs}[1]{\!\left| #1 \right|}
\newcommand{\abs}[1]{| #1 |}
\newcommand{\Norm}[1]{\!\left\lVert #1 \right\rVert}
\newcommand{\norm}[1]{\lVert #1 \rVert}
\newcommand{\spr}[2]{\langle #1 , #2 \rangle}
\newcommand{\Spr}[2]{\left\langle #1 , #2 \right\rangle}
\newcommand{\ISS}{\textls[50]{\textsc{iss}}}
\newcommand{\WCG}{\textls[50]{\textsc{wcg}}}
\newcommand{\iISS}{{\footnotesize i}\kern.03em{}\ISS{}}
\renewcommand{\epsilon}{\varepsilon}
\DeclareMathOperator{\id}{id}
\DeclareMathOperator{\ran}{ran}
\newcommand{\ball}{\mathrm{B}}
\newcommand{\wuc}{w.u.C.{}}
\newcommand{\uc}{u.c.{}}
\newcommand{\czero}{$\Cc_0$}
\setlist{itemsep=0.2ex, parsep=0.4ex}
\date{\today}
\author{\stepcounter{footnote}Philip Preußler\thanks{Mathematics of Systems Theory, Department of Applied Mathematics, University of Twente, P.\,O.\ Box 217, 7500~AE  Enschede, The Netherlands. E-mail: \texttt{p.n.preusler@utwente.nl}} \ and Felix L.\ Schwenninger\thanks{Mathematics of Systems Theory, Department of Applied Mathematics, University of Twente, P.\,O.\ Box 217, 7500~AE  Enschede, The Netherlands. E-mail: \texttt{f.l.schwenninger@utwente.nl}}}
\title{Implications of structured\\ continuous maximal regularity}
\begin{document}
\maketitle
\vspace{-4.5ex}
\begin{abstract}
    \noindent\textbf{Abstract.}
    We study how maximal regularity estimates with respect to the continuous functions improve automatically in cases where the spatial norm is fundamentally different from the supremum norm. 
    More precisely, we invoke properties such as weak compactness of convolution-type operators related to the mild solutions of the underlying linear evolution equations to sharpen the a priori estimates. These results have several applications: such as a new proof of Guerre-Delabriere's result on $\Ll^1$-maximal regularity and an extension of Baillon's theorem; a simplification for well-known perturbation theorems for generation of \czero-semigroups; and we resolve an open problem on input-to-state stability from control theory for a general abstract
    class of systems.
\end{abstract}\vspace{3ex}
\vspace{-2ex}
\section{Introduction}\label{sec:intro}
Well-posedness and regularity of (nonlinear) evolution equations, modeled via operator semigroups, are closely related to properties of underlying convolution-type operators that formally arise from the variation of constants formula. In this context, the interplay of temporal and spatial norms is crucial to apply, e.g., fixed-point arguments or perturbation results. In this paper we deal with the question whether a priori estimates with respect to limit case time norms, such as $\Ll^1$ or $\Ll^{\infty}$, can be sharpened to stronger estimates, taking into account structural assumptions on the geometry of the Banach space. Our abstract results further impact several domains in partial differential equations such as questions from control theory and maximal regularity.

To explain the setting, let $A$ generate a strongly continuous semigroup $T$ on a Banach space $X$. Operator semigroups have a decades-long tradition in the structural study of evolution equations; see, for example, \cite{pazySemigroupsLinearOperators1983,arendtVector-valuedLaplaceTransforms2011,engel_one-parameter_2000}. While the theory is fundamentally built up for linear problems given in the abstract form $\dot{x}(t)=Ax(t)$, the setting readily allows for the treatment of nonlinear equations; usually through rephrasing nonlinearities by means of fixed-point arguments. Yet, the properties of the linear dynamics are usually essential to apply these. 

Let $U$ be a Banach space, $B_{0}\colon U\to X$ be a bounded linear operator and let $\tau > 0$. Consider the singular convolution-type operator
\[\Phi\colon f \mapsto \Phi(f)\coloneq A \int_{0}^{\,\cdot\,}T(\,\cdot\,-s)B_{0}f(s)\dd s,\] 
which is (well-)defined for smooth $f\colon[0,\tau]\to U$. 
For $U=X$ and $B_{0}=\id_{X}$, the boundedness of $\Phi$ with respect to $\mathrm{Z}=\Ll^{p}([0,\tau];U)$, $p\in[1,\infty)$, or  $\mathrm{Z}=\Cc([0,\tau];U)$ refers to $\mathrm{Z}$-maximal regularity of the inhomogeneous equation 
\[\dot{x}(t)=Ax(t)+f(t),\qquad x(0)=0\in X,\ t\in [0,\tau].\]
Maximal regularity has seen particularly wide application in the study of nonlinear parabolic equations, see, e.g., \cite{dapratoEquationsdEvolutionAbstraites1979,denkRboundedness2003,prussMovingInterfacesQuasilinear2016,monniauxMaximalRegularity2024}, with a strong emphasis on $\Ll^{p}$-maximal regularity with $p\in(1,\infty)$. 
Here, we focus on \emph{structured} maximal regularity with respect to the \emph{continuous functions}; for short: \emph{structured continuous maximal regularity}; that is, the property of $A$ and $B_{0}$ that 
\[\Phi\colon \Cc([0,\tau];U)\to \Cc([0,\tau];X)\]
is a bounded operator. In this case we say that \emph{$A$ has $B_{0}$-structured $\Cc$-maximal regularity}. Our main interest lies in the question \emph{whether $\Phi$ extends boundedly to larger spaces} such as, e.g., $\Ll^{\infty}$. More precisely, we aim to find conditions on $U$ and $X$ (and possibly on $A$) such that $\Phi$ is bounded from some $\Ll^{p}([0,\tau];U)$ or some vector-valued Orlicz space $\Ee_F([0,\tau];U)$ to $\Cc([0,\tau];X)$. We emphasize that any such extension to $\Ll^p$ or $\Ee_F$ implicitly contains the assertion that
mild solutions to the above differential equation with inhomogeneities in $\Ll^p$ or $\Ee_F$ are continuous
functions with values in $X$.

At first glance, this problem may seem artificial, albeit interesting from an operator theoretic point of view. But already in the extremal case $U=X$ and $B_{0}=\mathrm{id}_{X}$, the situation is subtle. On the one hand, by Baillon's theorem \cite{baillonCaractere1980}, boundedness of $\Phi$ on $\Cc([0,\tau];X)$ implies that either $A$ is bounded or the space $X$ must contain an isomorphic copy of $c_{0}$. On the other hand, as was shown in \cite{jacobRefinementBaillonTheorem2022}, an extension to $\Ll^{\infty}$ (and hence to any other $\Ll^{p}$) is only possible if $A$ is bounded. Yet, the general case, where $B_{0}$ is not the identity, has a multitude of applications and was not addressed previously except in related special cases. The study of boundedness of $\Phi$ and related concepts are implicitly present in several fields of evolution equations, such as:
\begin{itemize}
    \item control of PDEs, where $f$ is, e.g., a (time-dependent) boundary datum on the spatial domain and $\Phi$ arises from a ``lifting'' of the boundary value problem to the domain,
\cite{tucsnak_observation_2009,mironchenkoInputtoStateStabilityInfiniteDimensional2020,staffans_well-posed_2005,haakKatoMethodNavier2009,haakWeightedAdmissibilityWellposedness2007,Jacob_2018,jacobContinuitySolutionsParabolic2019,byrnesRegularLinearSystems2002,arora2026inputtostatestabilityintegralnorms};
    \item perturbation results for semigroups, such as the Desch--Schappacher theorem \cite{deschSomeGenerationResults1989}, cf.\ also \cite{aroraLimitcaseAdmissibilityPositive2025a, barbieriPerturbationsPositiveSemigroups2025, barbieriOnStructuredPerturbations2025, batkaiPerturbationsPositiveSemigroups2018, wintermayrPositivityPerturbationTheory2019a,engel_one-parameter_2000,kunstmannPerturbationTheorems2001,adlerOnPerturbations2014};
    \item (continuous) maximal regularity, \cite{baillonCaractere1980,guerre-delabriere$L_p$regularityCauchyProblem1995, jacobRefinementBaillonTheorem2022, kaltonREMARKSL1MAXIMAL2008,lecroneContinuousMaximalRegularity2011,riMaximalL1Regularity2022,dapratoEquationsdEvolutionAbstraites1979, danchinFreeBoundaryProblems2025,lecroneOnquasilinear2020,hytonenAnalysisBanachSpaces2018,clementMaximalRegularity2001,arora2026inputtostatestabilityintegralnorms}, see also the references in \cite{lecroneContinuousMaximalRegularity2011};
     \item well-posedness of parabolic and hyperbolic semilinear equations (fixed-point arguments); where $U$ is typically
    a domain of some fractional power of the generator or another interpolation space \cite{pazySemigroupsLinearOperators1983, hosfeldInputToStateStabilityforBilinear2024, lecroneOnquasilinear2020,lunardi2018,dapratoEquationsdEvolutionAbstraites1979,hytonenAnalysisBanachSpaces2018,clementMaximalRegularity2001,danchinFreeBoundaryProblems2025}, see also the references in 
    \cite[Ch.\ 17.5]{hytonenAnalysisBanachSpaces2018}.
\end{itemize}
The common implicit feature among the above topics and the point of departure of this work is the following characterization of structured continuous maximal regularity
and the boundedness of the point evaluations of $\Phi$:
\begin{proposition}[Theorem \ref{thm:structuredmaxreg}]
        Let $\tau > 0$ and let $A$ generate a \czero-semigroup $(T(t))_{t\ge0}$ on the Banach space $X$. Let also $\delta_\tau$ be the usual evaluation map $f\mapsto f(\tau)$. If $U$ is another Banach space and $B_0 \in \LL(U,X)$, then $A$ has 
    $B_0$-structured $\Cc$-maximal regularity if and only if $\Phi_\tau \coloneq \delta_\tau\circ \Phi$, i.e., the map
    \[
        f \mapsto \Phi_{\tau} f = A\int_0^\tau T(\tau - s)B_0f(s)\dd s
    \]
    is a bounded map $\Cc([0,\tau];U) \to X$.
\end{proposition}
    The case $U = X$ and $B_0 = \id_X$ is found in Travis \cite[Prop.\ 3.1, Proof of Prop.\ 3.1]{travisDifferentiabilityWeakSolutions1981}; see also \cite[Prop.\ 2.2]{jacobRefinementBaillonTheorem2022}. In \cite[Cor.\ 5.4]{kruseContinuousMaximalRegularity2025}, the authors show this equivalence assuming that $B_0$ commutes with the semigroup. We provide a full proof of the general result in the appendix as Theorem \ref{thm:structuredmaxreg}. The proposition also shows the equivalence between structured maximal regularity and the notion of \emph{admissible operators} with respect to the continuous functions. Indeed, the property that the operator $\Phi_{\tau}$ is bounded means precisely that $A_{-1}B_{0}$ is $\Cc$-admissible, see Section \ref{sec:B_operators} for details.

We show that operators $\Phi$ from above do indeed extend to bounded operators with respect to stronger time norms if the
underlying Banach space $X$ fulfills certain geometric assumptions.
More precisely, we prove that under conditions such as $X$ 
being reflexive or $X$ having nontrivial Rademacher type, any $\Phi$ bounded on the continuous functions
extends to a bounded operator from some Orlicz space. These results
provide insight into and answer a series of open questions in the context of the aforementioned applications. For instance,
we show that the contractivity assumption $\limsup_{\tau\searrow 0}\norm{\Phi_{\tau}}<1$ in classical perturbation results on operator semigroups, cf., e.g., \cite[Cor.\ III.3.3]{engel_one-parameter_2000}, \cite{deschSomeGenerationResults1989}, is void for 
the previously mentioned classes of Banach spaces. On the 
other hand, we resolve a standing question on the equivalence
of integral input-to-state stability (for short: \iISS) and input-to-state stability (for short: \ISS) for a large class of
infinite-dimensional linear systems \cite{Jacob_2018,jacobContinuitySolutionsParabolic2019}. This result, which is
trivial in finite dimensions, had previously only been shown
for specific classes of semigroups \cite{jacobContinuitySolutionsParabolic2019}. Finally, our results
partially extend classical theorems due to Baillon \cite{baillonCaractere1980} and Guerre-Delabriere \cite{guerre-delabriere$L_p$regularityCauchyProblem1995} on maximal regularity with respect to the 
continuous functions resp.\ the absolutely integrable functions.
\subsection{Literature overview}
The problem of extending operators such as $\Phi$ has been treated previously in 
special situations. Structured maximal regularity as the main feature of the abstract setup and the generality
of our results are novel. In the context of control theory, there are preexisting results. We provide a summary of the previous work in this area from the literature.

For $U$ finite-dimensional and for analytic semigroups, generalizing earlier results from \cite{jacobOnInputToStateStability2016} and \cite{Jacob_2018} on diagonal semigroups, the article \cite{jacobContinuitySolutionsParabolic2019} considers such extensions from $\Ll^\infty$ to Orlicz spaces under
assumptions on the generator $A$, such as square function estimates or boundedness of 
the $\mathrm{H}^\infty$-calculus in the Hilbert space setting, see \cite[Prop.\ 6, Thm.\ 7, Cor.\ 8]{jacobContinuitySolutionsParabolic2019}. The equivalence of integral \ISS{} and \ISS{}
under these assumptions is also noted in \cite[Cor.\ 21]{jacobContinuitySolutionsParabolic2019}. 

In \cite{jacob2026laplacecarlesonembeddingsinfinitynormadmissibility},
the authors study boundedness of so-called Laplace--Carleson embeddings and the connection to mappings such as $\Phi$. In particular, for diagonal semigroups and finite-dimensional $U$, the authors show in
\cite[Thm.\ 3.6, Cor.\ 3.7]{jacob2026laplacecarlesonembeddingsinfinitynormadmissibility} that boundedness of $\Phi$ on $\Ll^\infty$ always implies boundedness of $\Phi$ on a suitable Orlicz space. Moreover, in \cite[Thm.\ 3.3]{jacob2026laplacecarlesonembeddingsinfinitynormadmissibility}, the authors even recover $\Ll^2$-admissibility
from $\Ll^\infty$-admissibility if $X$ and $U$ are Hilbert spaces and $A$ generates a \czero-\emph{group} on $X$; this is shown by applying previous results by Weiss and Hansen--Weiss.

The relation between admissibility and Baillon-type results was previously considered in \cite{jacobRefinementBaillonTheorem2022}. This includes the result \cite[Thm.\ 2.9]{jacobRefinementBaillonTheorem2022} that, in the case $U = X$ and $B_ 0=\id_X$, the operator $\Phi$, when assumed to be bounded on $\Cc([0,\tau];X)$, can only extend to $\Ll^\infty$ if the generator $A$ is bounded.

Results related to the equivalence of $B_0$-structured maximal continuous regularity and $\Cc$-admissibility have been shown in \cite[Cor.\ 5.4]{kruseContinuousMaximalRegularity2025}; 
however, under stronger assumptions such as commutation of $B_0$ with the semigroup. For $U = X$ and $B_0 = \id_X$, the setting of classical $\Cc$-maximal regularity, \cite{travisDifferentiabilityWeakSolutions1981} contains a proof 
of the equivalence.

The equivalence of $\Ll^\infty$-\ISS{} and $\Ll^\infty$-\iISS{} for parabolic diagonal systems with $B \in \LL(\C,X_{-1})$, $X$ a Hilbert space, has appeared in \cite[Thm.\ 3.1]{jacobOnInputToStateStability2016}. In \cite[Thm.\ 4.1, Rem.\ 4.3]{Jacob_2018} the same group of authors shows that the class of parabolic diagonal systems with respect to $q$-Riesz bases on Banach spaces $X$ has the property that $\Phi$ operators bounded on $\Ll^\infty$ always extend to some Orlicz space, and that the equivalence of \ISS{} and \iISS{} also holds here. Moreover, they also remark \cite[Rem.\ 4.3]{Jacob_2018}
that the $\Delta_2$ condition, a growth condition on the Young function associated to the
Orlicz space, can not be recovered in some cases; as this would already imply an extension 
to some $\Ll^p$.

Recently, results leading to admissibility with respect to the continuous functions (or $\Ll^\infty$) have been derived in 
the special context of positive semigroups. 
In \cite[Thm.\ 4.2.7]{wintermayrPositivityPerturbationTheory2019a} it was shown that 
that for Banach lattices $X$, $\dim U < \infty$ and positive $B\in\LL(U,X_{-1})$ \emph{zero-class} $\mathrm{Z}$-admissibility (cf.\ Section \ref{sec:B_operators} below) follows from $\mathrm{Z}$-admissibility for $\mathrm{Z} = \Cc$ or $\mathrm{Z} = \Ll^\infty$.

Also, in \cite[Thm.\ 4.5]{aroraLimitcaseAdmissibilityPositive2025a}
the authors show that zero-class admissibility with respect to $\Cc$ follows from $\Cc$-admissibility under assumptions related to the order of $X$ and an assumption on the range of $B$. Various 
sufficient conditions implying forms of $\Ll^\infty$- or $\Cc$-admissibility using different assumptions 
can also be found throughout \cite[Sect.\ 4.2]{aroraLimitcaseAdmissibilityPositive2025a}. Dropping
the order structure on $X$, the authors also obtain \cite[Thm.\ 4.12]{aroraLimitcaseAdmissibilityPositive2025a}, which is zero-class $\Ll^\infty$-admissibility from $\Ll^\infty$-admissibility when $X$ is reflexive and $U$ is an AM-space. Also, for $\dim U < \infty$ and $X$ reflexive,
\cite[Thm.\ 4.13]{aroraLimitcaseAdmissibilityPositive2025a} is the statement that $\Cc$-admissibility
always improves to zero-class $\Ll^\infty$-admissibility.
\subsection{Structure of the paper}
In the first section, we give a quick overview on notation and selected concepts that will be used throughout the paper.

Section 2 deals with a wide-ranging consequence of the classical de la 
Vallée Poussin theorem for linear operators $\Psi\colon W \to \Ll^1([0,1];Y)$ mapping into Lebesgue--Bochner
spaces. In particular, we note a folklore version of de la Vallée Poussin's theorem which states that uniform integrability of the image of the unit ball is necessary and sufficient for an operator, a priori bounded into an $\Ll^1$ space, to also be bounded with respect to a stronger Orlicz space norm. 
We comment on the application of weak compactness properties in this context, related to the Dunford--Pettis theorem. Also, we present several sufficient conditions on the involved Banach spaces
which yield the requisite uniform integrability property for any bounded operator.

The third section contains the main results of this article. The results are
phrased using the language of admissible operators in Banach spaces.
For operators $\Phi_\tau$ stemming from structured continuous maximal regularity as above, we use the framework of \emph{sun duality}, a duality
concept that is adapted to \czero-semigroups, to relate them to operators
mapping into $\Ll^1$ spaces. As $\Ll^{1}([0,\tau];U')$ is not the dual space of $\Cc([0,\tau];U)$, the duality arguments are subtle. The starting point rests on  \cite[Thm.\ 3.1]{aroraAdmissibleOperatorsSundual2025}, where the authors show that 
bounded $\Phi_\tau$ operators give rise to boundedness of a related, formally dual, operator
$\tilde\Psi$ defined on the sun dual space with values in a vector-valued 
$\Ll^1$ space. We then show that uniform integrability of the image under $\tilde\Psi$ of the unit ball in the sun dual space suffices to conclude
that the original operator $\Phi_\tau$ factorizes through an Orlicz space.

Moreover, we prove that weak compactness of $\Phi_\tau$ implies the uniform
integrability criterion for the associated formally dual operator. As in the
previous section, we discuss various other sufficient conditions, such as, again, reflexivity and nontrivial type.

In the final section, Section 4, we highlight applications
of the abstract results of Sections 2 and 3. We begin with admissibility for observation operators, arising from control theory, where we note that
the results of Section 2 directly translate into conditions where
$\Ll^1$-admissibility is automatically improved to Orlicz-admissibility. Next, we show that integral input-to-state stability is equivalent to input-to-state stability under the sufficient conditions of Section 3, which was an open problem for the last ten years. Then, 
we comment on the fact that an assumption on the perturbing operator can be dropped in the Desch--Schappacher theorem under assumptions on the Banach space. Further, we remark that Section 3 results lead to continuity of mild solutions for linear
systems with $\Ll^\infty$-admissible operators. We
also comment on the implications of positivity assumptions for Orlicz-admissibility and
note certain extensions of these results. As a final application, we present a novel 
proof of Guerre-Delabriere's result on $\Ll^1$-maximal regularity using the theory
from Section 2, while providing context on Baillon's theorem.

In the appendix, we collect some auxiliary statements.
\subsection{Terminology and notation}
Let $X$ be a Banach space. We write $X'$ for the dual space of $X$. If $x \in X$ and $x' \in X'$, we use any of the notations $\spr{x}{x'}_{X\times X'} \coloneq \spr{x'}{x}_{X'\times X} \coloneq x'(x)$ to refer
to the dual pairing. The closed unit ball of $X$ is denoted by
$\ball_X \coloneq \{ x \in X : \norm{x}_X \leq 1 \}$. Recall also that a subspace $M \subset X$ is called complemented if
there is a projection $P_M \in \LL(X)$ with $\ran P_M = M$, see, e.g.,
\cite[37]{wojtaszczykBanachSpacesAnalysts1991}.

If $Y$ is a second Banach space, we use the symbol 
$\LL(X,Y)$ to stand for the bounded linear operators from $X$ to $Y$. Recall that an
operator $T \in \LL(X,Y)$ is called weakly compact if the image $T(\ball_X)$ is a relatively
weakly compact subset of $Y$; cf., e.g., \cite[Def.\ VI.5.1]{conwayCourseFunctionalAnalysis2007}.
We say that an operator $T\in\LL(X,Y)$ factors through a third Banach space $W$ if there
are operators $Q\in \LL(X,W)$ and $R\in \LL(W,Y)$ such that $T = RQ$. 

Also, if $A$ is a (possibly 
unbounded) operator on $X$, then $R(\lambda,A) \coloneq (\lambda \id_X - A)^{-1}$ denotes the resolvent 
of $A$ for parameters $\lambda \in \rho(A)$, the resolvent set of $A$. If $A\colon X \supset D(A) \to X$ is the generator of a \czero-semigroup on $X$, then we use the shorthand $X_1$ to denote the
domain $D(A)$ equipped with the graph norm.

For a Banach space $Y$ and an interval $I \subset \R$ and $p \in [1,\infty]$, we denote by $\Ll^p(I;Y)$ the usual Lebesgue--Bochner space.
Lastly, we point out that we use $[0,1]$ generically as a bounded interval for notational convenience. The results
can also be stated for any other bounded interval with left endpoint $0$.
\section{Uniform integrability and linear maps into $\Ll^1$}\label{sec:operatorsintoL1}
Let $W$ and $Y$ be Banach spaces. We discuss the relation between \emph{uniform
integrability} of subsets of $\Ll^1$ spaces and conditions that ensure that
operators $\Psi \in \LL(W, \Ll^1([0,1];Y))$ factorize through a function space $\mathrm{Z}([0,1];Y) \subset
\Ll^1([0,1];Y)$ which is equipped with a stronger norm. We focus on the Orlicz function spaces; see Section \ref{sec:orlicz} in the sequel.
\subsection{Uniform integrability}\label{sec:ui}
Following \cite[101]{diestelVectorMeasures1977}, we say that $M \subset \Ll^1([0,1];Y)$ is called \emph{uniformly
    integrable} if for any sequence $(E_n)_{n\in\N}$ of 
    measurable subsets of $[0,1]$ with $\abs{E_n} \to 0$ we have 
    \[
    \lim_{n\to\infty} \sup_{f\in M} \int_{E_n} \norm{f(t)}_Y \dd t = 0.
    \]
Since the Lebesgue measure is non-atomic, 
uniform integrability of $M$ implies that $M$ is bounded in $\Ll^1$ norm; see, e.g., \cite[112]{albiacTopicsBanachSpace2016}. Since bounded subsets of $\Ll^1$ need not be uniformly integrable,
cf., e.g., \cite[Ex.\ 5.2.3 (iii)]{albiacTopicsBanachSpace2016}, uniform integrability is a 
strictly stronger property.
    
Due to the classical Dunford--Pettis theorem, there is a deep link between uniform integrability and
weak compactness in $\Ll^1$. For scalar $\Ll^1$ spaces, the theorem asserts that a subset $M$
of $\Ll^1([0,1])$ is relatively weakly compact
if and only if $M$ is uniformly integrable. See, e.g., \cite[Thm.\ 5.2.8]{albiacTopicsBanachSpace2016} or \cite[Thm.\ III.2.15]{diestelVectorMeasures1977} for  proofs. On passing to vector-valued spaces, the two properties are no longer equivalent. Uniform integrability of $M \subset \Ll^1([0,1];Y)$ 
still follows from relative weak compactness of $M$; see, e.g., \cite[Thm.\ IV.2.4]{diestelVectorMeasures1977} and the discussion on \cite[104]{diestelVectorMeasures1977}.
For general $Y$ the converse fails and additional
assumptions are needed to recover relative weak compactness of uniformly integrable sets.
See also \cite[p.~101~ff.]{diestelVectorMeasures1977}, \cite[p.~62~ff.]{diestelUniformIntegrabilityIntroduction1991} and the references
\cite{ulgerWeakCompactnessL11991, diestelWeakCompactnessL1muX1993, Diestel_1977} mentioned in \cite{diestelUniformIntegrabilityIntroduction1991} for more on weak compactness
in $\Ll^1([0,1];Y)$.
\subsection{Sufficient conditions}
In this section, we are concerned with uniform integrability of sets of the form $\Psi(\ball_W)$, where
$\Psi$ is a bounded operator $W \to \Ll^1([0,1];Y)$. Reflexivity of $W$ is a central sufficient condition:
\begin{proposition}\label{prop:reflexiveui}
  Let $W$ be a reflexive Banach space and $Y$ be another Banach space.
  If $\Psi \in \LL(W,\Ll^1([0,1];Y))$, then $\Psi$ is weakly compact and $\Psi(\ball_W)$
  is uniformly integrable.
\end{proposition}
\begin{proof}
Since $\Psi(\ball_W)\subset \Ll^1([0,1];Y)$ is relatively weakly compact,
by the previously mentioned Dunford--Pettis type results (see, e.g., \cite{diestelWeakCompactnessL1muX1993}),
    $\Psi(\ball_W)$ is uniformly integrable.
\end{proof}
\subsubsection{(V*) sets, complemented copies of $\ell^1$, Rademacher type}
The aim of this section is to show that if $\Psi \in \LL(W,\Ll^1([0,1];Y))$, uniform integrability of $\Psi(\ball_W)$ 
is automatic when the Banach space $W$
does not contain the space $\ell^1$ as a complemented subspace; this seems to be a folklore result. In the proof, we use results from 
Bombal's article \cite{bombalSetsPelczynskisProperty1990} on Pełczyński's property (V*).
Pełczyński's property (V*) originates from Pełczyński's article \cite{pelczynskiBanachSpacesWhich1962}
as a property dual to the property (V); see Section \ref{sec:prop(V)} below. 

To introduce the necessary terminology for property (V*), we recall the definition of weakly unconditionally Cauchy series. Following \cite[{}{II.D.3}]{wojtaszczykBanachSpacesAnalysts1991}, for $(w_n)_{n\in\N} \subset W$ we say that the (formal) series $\sum_{n\in\N} w_n$ is weakly
unconditionally Cauchy (\wuc), if for each $w' \in W'$ the series
$\sum_{n\in\N} w'(w_n)$ is an unconditionally convergent series in $\C$. See also \cite[Def.\ 2.4.3]{albiacTopicsBanachSpace2016}, \cite[641]{pelczynskiBanachSpacesWhich1962},
\cite[Prop.\ 4.3.9 ff.]{megginsonIntroductionBanachSpace1998}, 
\cite[p.\ 58 ff.]{wojtaszczykBanachSpacesAnalysts1991}, \cite[p.\ 99 ff.]{lindenstraussClassicalBanachSpaces1977} and \cite[Ch.\ 2.4]{albiacTopicsBanachSpace2016} for more on \wuc{} series.

Following the definition in \cite[109]{bombalSetsPelczynskisProperty1990}, see also \cite[171]{emmanueleBanachSpacesProperty1988}, a subset $M$ of a Banach
space $W$ is called a (V*) set if, for every \wuc{}\ series $\sum_{n\in\N} w_n'$ in $W'$, it 
holds that 
\[
    \lim_{n\to\infty} \sup_{w\in M} \abs{w_n'(w)} = 0.
\]  
The space $W$ is said to have Pełczyński's property (V*) if (V*) sets in $W$ are weakly compact, see \cite[110]{bombalSetsPelczynskisProperty1990} and compare \cite[Def.\ 3]{pelczynskiBanachSpacesWhich1962}, where
the term (V*) set is not used.

However, we do not make further use of property (V*) itself; the argument uses (V*) sets. With Bombal's results from \cite{bombalSetsPelczynskisProperty1990}, we provide a proof of the following folklore result:
\begin{proposition}
\label{prop:complementedforui}
    Let $W$ be a Banach space containing no complemented copy of $\ell^1$.
    If $Y$ is a Banach space and $\Psi\in \LL(W,\Ll^1([0,1];Y))$,
    then $\Psi(\ball_W)\subset \Ll^1([0,1];Y)$ is uniformly integrable.
\end{proposition}
\begin{proof}
    Let $\Psi \in \LL(W,\Ll^1([0,1];Y))$. 
    Since $W$ lacks a complemented subspace isomorphic to $\ell^1$, \cite[Cor.\ 1.5]{bombalSetsPelczynskisProperty1990} implies that bounded
    subsets of $W$ are (V*) sets; hence $\ball_W$ is a (V*) set.
    Using (b) on \cite[110]{bombalSetsPelczynskisProperty1990}, which states that bounded operators
    take (V*) sets to (V*) sets,
    it follows that $\Psi(\ball_W)$ is a (V*) set in $\Ll^1([0,1];Y)$. Lastly, \cite[Prop.\ 3.1]{bombalSetsPelczynskisProperty1990} implies that (V*) sets in $\Ll^1([0,1];Y)$ are uniformly integrable, ending the proof.
\end{proof}
Building on this result and the notion of (Rademacher) type, we get a further sufficient condition.
We give an abridged summary of the definition of type, as we do not need the concept in its full generality; for 
the full definition of type $p\in [1,2]$ and cotype $q \in [2,\infty]$ involving
Rademacher random variables, we refer to \cite[Def. 7.1.1]{hytonenAnalysisBanachSpaces2018}.

For our purposes, it suffices to distinguish spaces with trivial type 1 from those with nontrivial type $p > 1$.
We use the Maurey--Pisier characterization, see, e.g., \cite[Thm.\ 7.3.8]{hytonenAnalysisBanachSpaces2018}, as a definition.

By \cite[Def.\ 7.1.8, Def.\ 7.3.7, Thm.\ 7.3.8]{hytonenAnalysisBanachSpaces2018}, we say that the Banach space $W$
has trivial type if $\ell_n^1$ is contained in $W$ uniformly in $n$, i.e., there is $\lambda \geq 1$ such that 
for each $n\in \N$ there is an operator $T_n\in \LL(\ell^1_n, W)$ and $\lambda_{1,n},\lambda_{2,n} > 0$ with $\lambda_{2,n}/\lambda_{1,n} \leq \lambda$ and
the property that 
\[\lambda_{1,n}\norm{w^{(n)}}_1 \leq \norm{T_n w^{(n)}}_W \leq \lambda_{2,n}\norm{w^{(n)}}_1\]
holds for every $w^{(n)} \in \ell^1_n$.  

Similarly, also by \cite[Def.\ 7.1.8, Def.\ 7.3.7, Thm.\ 7.3.8]{hytonenAnalysisBanachSpaces2018},
we say that $W$ has cotype $\infty$ if $\ell^\infty_n$ is uniformly contained
in $W$.
\begin{remark}
By \cite[Cor.\ 7.4.24]{hytonenAnalysisBanachSpaces2018}, the space $W$ has nontrivial type if and only if $W'$ has nontrivial type.
Moreover, if $W$ has nontrivial type, then $W$ also has finite cotype, see \cite[Prop.\ 7.1.14]{hytonenAnalysisBanachSpaces2018}. Lastly, the definition implies that spaces with
finite cotype cannot contain $c_0$.

We recall that reflexivity and nontrivial Rademacher
type of Banach spaces are logically independent. Examples can be constructed of reflexive spaces 
that only have trivial type, see, e.g., \cite[318--319]{albiacTopicsBanachSpace2016} and of 
spaces with nontrivial type that
are not reflexive; see, e.g., \cite{jamesNonreflexive1978}.  
        
However, for many classes of well-known function spaces, such as the $\Ll^p$ scale, Orlicz spaces $\Ll_F$, and closed
subspaces of $\Ll^1$, the concepts of nontrivial type and reflexivity \emph{are} equivalent. For
$\Ll^p$ spaces see, e.g., \cite[Cor.\ 11.7]{diestelAbsolutelySummingOperators1995}.
For Orlicz spaces, the assertion follows from \cite[Thm.\ 4.13.9]{pickFunctionSpaces2013}
in conjunction with \cite[Prop.\ 4, Thm.\ 5]{Kaminiska_Turett_1991}; see also
\cite[Ch.\ 11, Notes and Remarks]{diestelAbsolutelySummingOperators1995}.
Lastly, for closed subspaces of $\Ll^1([0,1])$ the claim follows from \cite[Thm.\ 13.10, Thm.\ 13.19]{diestelAbsolutelySummingOperators1995}.
\end{remark}
\begin{corollary}\label{cor:typeforCs}
    Assume $W$ has nontrivial type. Then every $\Psi\in\LL(W, \Ll^1([0,1];Y))$
    has the property that
    $\Psi(\ball_W)$ is uniformly integrable.
\end{corollary}
\begin{proof}
        If $W$ has nontrivial type, any subspace of
        $W$ has nontrivial type, see \cite[219]{diestelAbsolutelySummingOperators1995}. Since $\ell^1$
        has trivial type \cite[Cor.\ 7.1.10]{hytonenAnalysisBanachSpaces2018} and type is preserved under
        isomorphisms \cite[Rem.\ 6.2.11\ (f)]{albiacTopicsBanachSpace2016}, it follows that no subspace
        of $W$ is isomorphic to $\ell^1$. In particular, $\ell^1$ is not isomorphic
        to any complemented subspace of $W$. 
        By Proposition
        \ref{prop:complementedforui} above, $\Psi(\ball_W)$
        is uniformly integrable.
\end{proof} 
\begin{remark}
    Parenthetically, we note that Corollary \ref{cor:typeforCs} can alternatively be
    proven directly using a Maurey--Nikishin-type
    factorization result for sublinear operators from \cite[Cor.\ 2.1]{diestelFactoringMultisublinearMaps2014}. 
\end{remark}
\subsection{Boundedness in Orlicz spaces and de la Vallée Poussin's
criterion}\label{sec:orlicz}\label{sec:OrliczIntro}
In our context, the main benefit of uniform integrability is that de~la Vallée~Poussin's theorem applies, asserting that a set $M \subset \Ll^1([0,1];Y)$
is uniformly integrable if and only if there is a convex function $F$ such that 
$F\circ f$ is integrable for all $f\in M$.
Moreover, 
this integrability property can be seen as boundedness in some Orlicz space $\Ll_F$. 
\subsubsection{Orlicz spaces}
In introducing the Orlicz function spaces, we follow Section 4 in \cite[Sect.\ 4]{pickFunctionSpaces2013} and \cite{hosfeldInputtostateStabilityClasses2025}. Compare also \cite[Ch.\ 8]{adamsSobolevSpaces2003}, \cite{raoTheoryOrliczSpaces1991} and \cite{krasnoselskiiConvexFunctionsOrlicz1961}. 
We begin with Young functions; we note that, in literature, the definitions of 
Young functions vary, and that some authors use the term ``N-function''.

Following \cite[Definition 4.2.1, Remark 4.2.7]{pickFunctionSpaces2013}, we call $F\colon [0,\infty) \to [0,\infty)$ a \emph{Young function} if $F$ is convex, continuous, and increasing with $\lim_{t\searrow 0} F(t)/t = 0$ and $\lim_{t\to\infty} F(t)/t = \infty$.

In the rest of this section
we let $F$ be a fixed Young function, $[a,b]$ a fixed interval and $Y$ be a
given Banach space.
As in \cite[Def.\ III.1.5]{raoTheoryOrliczSpaces1991}, the \emph{Orlicz space} $\Ll_F([a,b];Y)$ is defined by
\[
    \Ll_F([a,b];Y) \coloneq \biggl\{ f\in\Ll^1([a,b];Y) : \exists c>0\colon \int_a^b F(c\norm{f(x)}_Y)\dd x < \infty \biggr\}.
\]
We use the Luxemburg norm on $\Ll_F([a,b];Y)$; following \cite[Definition 4.8.1]{pickFunctionSpaces2013}, we set 
\[
    \norm{f}_{\Ll_F([a,b];Y)} \coloneq \inf\biggl\{ k > 0 : \int_a^b F(\norm{f(s)}_Y/k)\dd s \leq 1 \biggr\}.
\]
This norm makes $\Ll_F([a,b];Y)$ into a Banach space; see, e.g.,
\cite[Thm.\ III.3.10]{raoTheoryOrliczSpaces1991}, \cite[Cor.\ 1.2.22]{hosfeldInputtostateStabilityClasses2025} or \cite[{}{(5.1.22)}]{edgarStoppingTimesDirected1992}. We also introduce the space $\Ee_F([a,b];Y) 
\subset \Ll_F([0,1];Y)$, which we define as the closure of the $Y$-valued
simple functions; see \cite[Def.\ 4.12.1, Thm.\ 4.12.8]{pickFunctionSpaces2013}.

If $F$ is a Young function, we call the function $\tilde F$ defined by $\tilde F(x) \coloneq \sup \{ xy - F(y) : y \geq 0 \}$
the complementary (Young) function corresponding to $F$; see \cite[6]{raoTheoryOrliczSpaces1991}.

Hölder's inequality for $f\in \Ll_F([a,b];Y)$ and $g\in \Ll_{\tilde F}([a,b];Y')$ reads as follows:
\[
    \int_a^b \abs{\spr{f(x)}{g(x)}_{Y\times Y'}} \dd x \leq 2 \norm{f}_{\Ll_F([a,b];Y)}\norm{g}_{\Ll_{\tilde F}([a,b];Y')},
\]
see, e.g., \cite[Lem.\ 1.2.19]{hosfeldCharacterizationOrliczAdmissibility2023}, \cite[Prop.\ III.3.1, Rem.\ p.\ 58]{raoTheoryOrliczSpaces1991},
\cite[Thm.\ 4.8.5, Thm.\ 4.8.7]{pickFunctionSpaces2013}.

In some cases it can be useful to alternatively consider Orlicz'
norm, which we denote by $\norm{\,\cdot\,}_{\Ll_F([a,b];Y),\text{O}}$, see, e.g.,
\cite[Def.\ 4.6.1]{pickFunctionSpaces2013}. 
Importantly, the
Orlicz and Luxemburg norms are equivalent, see \cite[Prop.\ III.3.4]{raoTheoryOrliczSpaces1991}.

We say that $F$ satisfies the $\Delta_2$-condition (denoted $F\in \Delta_2$, see, e.g.\ \cite[Def.\ 4.4.1]{pickFunctionSpaces2013}) if there is $t_0 \geq 0$ and $C > 0$ such that $F(2t) \leq C F(t)$ holds for all $t \geq t_0$. Note that $\Ee_F([0,1];Y) = \Ll_F([0,1];Y)$ if and only if $F \in \Delta_2$; cf.\ \cite[Prop.\ 4.12.3, Rem.\ 4.12.4]{pickFunctionSpaces2013}.
\subsubsection{De la Vallée Poussin's theorem}
Let $Y$ be a Banach space. It can be shown that, individually, each element
of $\Ll^1([0,1];Y)$ lies in some Orlicz space, cf., e.g., \cite[Theorem~4.2.5]{pickFunctionSpaces2013}.
However, an arbitrary bounded set $M \subset \Ll^1([0,1];Y)$ need not be contained in a single Orlicz space. But under a suitable stronger condition, which, due to the classical de~la Vallée~Poussin theorem, is uniform
integrability of $M$, containment in a shared Orlicz space does follow. 

There is a great variety of work on uniform integrability and de la Vallée Poussin-type theorems; we mention the references \cite{chafaiValleePoussinUniform2014,diestelUniformIntegrabilityIntroduction1991,diestelVectorMeasures1977,raoTheoryOrliczSpaces1991,dellacherieProbabilitiesPotential1978,meyerLemmeValleePoussin1978,edgarStoppingTimesDirected1992} in this context.
We will make use of the following ``strengthened'' version of de la Vallée Poussin's theorem. This formulation seems to be folklore; see, e.g., the proof of \cite[Thm.\ 5.1]{ciprianiSobolevOrliczImbeddings2000}. Using \cite[Thm.\ II.22]{dellacherieProbabilitiesPotential1978}, \cite[Lemme]{meyerLemmeValleePoussin1978}, \cite[Thm.\ I.2.2, Lem.\ III.4.1]{raoTheoryOrliczSpaces1991} and \cite[Thm.\ 2.3.5]{edgarStoppingTimesDirected1992},
a full proof of this version of the theorem can be obtained.
\begin{proposition}[de la Vallée Poussin]
    Let $M$ be a bounded subset of $\Ll^1([0,1];Y)$. Then the following are equivalent:
    \begin{enumerate}[label=\alph*\textup{)}]
        \item there is a Young function $G \in \Delta_2$ such that 
            $M \subset \Ee_G([0,1];Y)$ is norm bounded;
        \item there is a Young function $F$ such that $M$ 
        is a norm bounded subset of $\Ll_F([0,1];Y)$;
        \item $M$ is uniformly integrable.
    \end{enumerate}
\end{proposition}
We observe that this theorem has the following consequence for 
linear operators into $\Ll^1$:
\begin{theorem}\label{thm:ui_C}
    Let $W$ and $Y$ be Banach spaces and let $\Psi\colon W \to \Ll^1([0,1];Y)$ be linear. The following are equivalent:
    \begin{enumerate}[label=\alph*\textup{)}]
        \item There is a Young function $G \in \Delta_2$
        such that $\Psi$ factors through $\Ee_G([0,1];Y)$ with 
        \[\norm{\Psi w}_{\Ll_G([0,1];Y)} \leq C_G \norm{w}_W\] for $w\in W$ and some $C_G \geq 0$;
        \item There is a Young function $F$
        such that $\Psi$ factors through 
        $\Ll_F([0,1];Y)$ with \[\norm{\Psi w}_{\Ll_F([0,1];Y)} \leq C_F \norm{w}_W\] for $w\in W$ and some $C_F \geq 0$;
        \item The set $\Psi(\mathrm{B}_W)
        \subset \Ll^1([0,1];Y)$
        is uniformly integrable.
    \end{enumerate}
\end{theorem}
\begin{proof}
The statements a) and b) are equivalent to $\Psi(\ball_W)$ being a norm bounded subset of
the corresponding Orlicz space. Then the claim follows from de la Vallée Poussin's 
theorem.\end{proof}
\begin{corollary}\label{cor:reflexiveforC}\label{cor:complementedl1Orlicz}
Let $W$, $Y$ be Banach spaces.
If the operator $\Psi\colon W \to \Ll^1([0,1];Y)$ is weakly compact, 
there exists a Young
    function $G \in \Delta_2$ such that $\Psi$ factors through
    $\Ee_G([0,1],Y)$. 
\end{corollary}
\begin{proof}
    The claim follows by the Dunford--Pettis theorem and de la Vallée Poussin's theorem.
\end{proof}
\section{Main abstract results}\label{sec:B_operators}
To introduce our main results, which rely on the structure of the operators $\Phi$ as introduced in Section \ref{sec:intro}, we pass to the framework of admissible operators. For this, we let $X$ be a Banach space and let $A$ be the generator of a \czero{}-semigroup $T$
on $X$. We consider operators $B \in \LL(U, X_{-1})$; where $U$ is another 
Banach space and $X_{-1}$ is the extrapolation space obtained by completing $X$ in the norm $\norm{R(z,A)\,\cdot\,}_X$ for some $z\in\rho(A)$. 

The operator $B$ is 
called $\mathrm{Z}$-admissible if
$\Phi_1\colon u \mapsto \int_0^1 T(1-s)Bu(s)\dd s$
is a well-defined and bounded operator $\mathrm{Z}([0,1];U) \to X$, where $\mathrm{Z}$ is a placeholder for the spaces $\Cc$, $\Ll^{p}$ for $p\in [1,\infty]$, or $\Ee_{F}$ for some Young function $F$.
We say that $B$ is zero-class $\mathrm{Z}$-admissible if even 
$ \sup_{u\in \mathrm{Z}([0,\epsilon];U) \leq 1} \norm{\int_0^\epsilon T(\epsilon - s) Bu(s) \dd s}_{X}
    \to 0$ as $\epsilon \searrow0$.
Admissible operators have been studied intensively in the past decades \cite{weissAdmissibleObservationOperators1989,weiss_admissibility_1989,staffans_well-posed_2005} in the context of infinite-dimensional control theory. One typically distinguishes \emph{admissible control}---we simply use the term \emph{admissible} here---and \emph{observation} operators; see also Section \ref{sec:admissibilityobservation} in the sequel.

In this section, we investigate the relation between $\Cc$-admissibility and Orlicz-admissibility. Before we state the main results, we revisit the connection between
$\Cc$-admissibility and structured maximal regularity with respect to the continuous 
functions, which was mentioned in the introduction of this article.
\begin{proposition}
Let $X$ and $U$ be Banach spaces and let $A$ generate a \czero-semigroup on $X$. Let $B \in \LL(U,X_{-1})$ and set $B_0 \coloneq R(\lambda,A)B$ for some $\lambda \in \rho(A)$. The following statements are equivalent:
\begin{enumerate}
    \item $A$ has $B_0$-structured maximal $\Cc$-regularity; \item $\Phi\colon f \mapsto A\int_0^{\,\cdot\,} T(\,\cdot\, - s) B_0 f(s) \dd s$ is a bounded map $\Cc([0,1];U) \to \Cc([0,1];X)$;
    \item $B$ is $\Cc$-admissible for $A$;
    \item $\Phi_1\colon f \mapsto A\int_0^{1} T(1 - s) B_0 f(s) \dd s$ is a bounded map $\Cc([0,1];U) \to X$;
\end{enumerate}
\begin{proof}
The statements 1.\ $\Leftrightarrow$ 2.\ and 3.\ $\Leftrightarrow$ 4.\ follow from the definitions. The remaining
implication 2.\ $\Leftrightarrow$ 3.\ is the content of
Theorem \ref{thm:structuredmaxreg} in the appendix.
\end{proof}
\end{proposition}
\begin{remark}
    By \cite[Proof of Prop.\ 5]{jacobContinuitySolutionsParabolic2019}, Orlicz-admissibility implies 
    zero-class $\Ll^\infty$-admissibility and thus also zero-class $\Cc$-admissibility. Also, it is known that $\Cc$-admissibility is equivalent to admissibility with respect
    to the space of regulated functions: see, e.g., \cite[Prop.\ 2.1]{aroraAdmissibleOperatorsSundual2025}. See also \cite[Thm.\ 10.2.2]{staffans_well-posed_2005}
    for more on regulated admissibility.
\end{remark}
Using sun duality, which is a duality concept that takes the semigroup structure into account, we present sufficient conditions for the improvement of $\Cc$-admissibility to Orlicz-admissibility. The main
argument rests on identifying the image of a formal dual map related to 
$\Phi_1$ with a uniformly integrable subset of
$\Ll^1([0,1];U')$ and then applying the results of Section 2.
\subsection{Sun duality and uniform integrability}
It is well-known that the adjoint semigroup $T'$ on the dual space $X'$ is not necessarily strongly continuous, see, e.g., \cite[II.2.6, Examples, pp.~63--64]{engel_one-parameter_2000} and \cite[Ex.\ 1.3.9]{neervenAdjointSemigroupLinear1992}. Following
\cite[5]{neervenAdjointSemigroupLinear1992} (see also \cite[II.2.6, pp.~62--64]{engel_one-parameter_2000}), we define the sun dual space $X^\odot$ as the
space of strong continuity of the semigroup $T'$ (in particular, $X^\odot$ depends on $T$) by setting
\[
    X^\odot \coloneq \{ x' \in X' : \norm{T'(t) x' - x'}_{X'} \to 0\text{ as }t\searrow 0\}.
\]
By \cite[5]{neervenAdjointSemigroupLinear1992}, $X^\odot$ is a closed subspace of $X'$,
hence itself a Banach space. Note that, for any semigroup $T$ and any Banach space $X$, the domain of $A'$ is a subset of $X^\odot$, see, e.g., \cite[II.2.6, Lem., p.~62]{engel_one-parameter_2000}. Next, by \cite[6]{neervenAdjointSemigroupLinear1992}, $T^\odot \coloneq T|_{X^\odot}$ is a \czero-semigroup and by \cite[Thm.\ 1.3.3]{neervenAdjointSemigroupLinear1992}, cf.\ also \cite[II.2.6, Prop., p.~63]{engel_one-parameter_2000}, its generator $A^\odot$ is given by 
\[
    A^\odot x^\odot \coloneq A'x^\odot,\qquad x^\odot \in D(A^\odot) \coloneq \{x^\odot \in  D(A') : A'x^\odot \in X^\odot\}.
\]
By \cite[Thm.\ 3.1.15, Cor.\ 3.1.17]{neervenAdjointSemigroupLinear1992}, sun duality obeys the relation $(X_{-1})^\odot = D(A^\odot)$.

In the context of the semigroup $T$ on $X$ as above, if we now consider another Banach space $U$ and suppose that $B \in \LL(U,X_{-1})$ is $\Cc$-admissible, then the input operator $\Phi_1\colon
\Cc([0,1];U) \to X$ is bounded. The following result from \cite[Thm.\ 3.1]{aroraAdmissibleOperatorsSundual2025}, which will be crucial in the next section, shows the connection between $\Cc$-admissibility and operators $X^\odot \to \Ll^1([0,1];U')$:
\begin{proposition}[{\cite[Thm.\ 3.1]{aroraAdmissibleOperatorsSundual2025}}]\label{prop:sun_duality}
If $X$ and $U$ are Banach spaces and $T$ is a \czero-semigroup on $X$, then an
operator $B \in \LL(U,X_{-1})$ is $\Cc$-admissible if and only
if
\[
    \tilde \Psi_1\colon X^\odot \to \Ll^1([0,1];U'),\qquad x^\odot \mapsto \left.B'\right|_{(X^\odot)_1}T^\odot (\,\cdot\,) x^\odot
\]
defines a bounded linear operator.
\end{proposition}
In a similar way, we connect uniform integrability of $\tilde{\Psi}_{1}(\ball_{X^{\odot}})$ to Orlicz-admissibility.
\begin{theorem}\label{thm:equi_for_B}
Let $X$ be a Banach space and let $T$ be a \czero-semigroup on $X$. Let 
$U$ be another Banach space and let $B \in \LL(U,X_{-1})$ be $\Cc$-admissible.
Consider the following statements:
    \begin{enumerate}[label=\alph*\textup{)}]
        \item There is a Young function $F$ such that $B$ is $\Ee_F$-admissible;
        \item The set $B'|_{(X^{\odot})_1} T^{\odot}(\,\cdot\,) \mathrm{B}_{X^\odot}$
        is a uniformly integrable subset of $\Ll^1([0,1];U')$.
    \end{enumerate}
    Then b\textup{)} implies a\textup{)}. If $U'$ has the Radon--Nikodým property, then a\textup{)} implies b\textup{)}.
\end{theorem}
\begin{proof}
We first show that b)\ implies a). To that end, let $B'|_{(X^\odot)_1} T^\odot(\,\cdot\,)\mathrm{B}_{X^\odot}$ be uniformly integrable. 
    For each $u \in \Cc([0,1];U)$, due to admissibility of $B$,
    $\Phi_1 u = \int_0^1 T(1-s) Bu(s) \dd s$ is a well-defined element of $X$. Proceeding as in the proof of \cite[Thm.\ 3.1]{aroraAdmissibleOperatorsSundual2025}, we consider $u \in \Cc([0,1];U)$ and 
    $x^\odot \in D(A^\odot)$ and compute
    \begin{align*}
        \abs{\spr{\Phi_1 u}{x^\odot}_{X\times X'}} &= \Abs{\Spr{\int_0^1 T(1-s)Bu(s) \dd s}{x^\odot}_{X\times X'}}\\
        &= \Abs{\int_0^1{\spr{u(1-s)}{B'|_{(X^\odot)_1}T^{\odot}(s)x^\odot}}_{U\times U'}\dd s}.
    \end{align*}
    Note that $T^\odot(s) x^\odot \in (X^\odot)_1 = D(A^\odot)$ for $x^\odot \in D(A^\odot)$ by the strong continuity of the sun dual semigroup $T^\odot$, see, e.g., \cite[Lem.\ 1.3~(ii)]{engel_one-parameter_2000}.

    By Theorem \ref{thm:ui_C} above, uniform integrability of $B'|_{(X^{\odot})_1} T^{\odot}(\,\cdot\,) \mathrm{B}_{X^\odot}$ leads to the existence of a Young function $\tilde F$ such that 
    \[
        M\coloneq\sup_{x^\odot\in \ball_{X^\odot}}\norm{B'|_{(X^\odot)_1}T^\odot(\,\cdot\,)x^\odot}_{\Ll_{\tilde F}([0,1];U')}< \infty.
    \]
    Since $u$ is a continuous function with values in $U$, $u$ lies in any $U$-valued Orlicz space
    on $[0,1]$. In particular, if $F$ is the complement of $\tilde F$ from above, it follows
    that $u \in \Ee_F([0,1];U)$.
    
    An application of the generalized
    Hölder inequality 
    (see, e.g., \cite[Lem.\ 1.2.19]{hosfeldCharacterizationOrliczAdmissibility2023}, \cite[Prop.\ III.3.1, Rem.\ p.\ 58]{raoTheoryOrliczSpaces1991},
\cite[Thm.\ 4.8.5, Thm.\ 4.8.7]{pickFunctionSpaces2013}) 
     yields
    \begin{align*}\Abs{\int_0^1{\spr{u(1-s)}{B'|_{(X^\odot)_1}T^{\odot}(s)x^\odot}}_{U \times U'}\dd s} &\leq 2 \norm{u}_{\Ll_F([0,1];U)}\norm{B'|_{(X^\odot)_1}T^{\odot}(\,\cdot\,)x^\odot}_{\Ll_{\tilde F}([0,1];U')}.
    \end{align*}
    Hence
    \[
    \abs{\spr{\Phi_1u}{x^\odot}_{X\times X'}} \leq 2 \norm{u}_{\Ll_F([0,1];U)}\norm{B'|_{(X^\odot)_1}T^{\odot}(\,\cdot\,)x^\odot}_{\Ll_{\tilde F}([0,1];U')}
    \]
    holds for $u \in \Cc([0,1];U)$ and 
    $x^\odot \in D(A^\odot)$. 
    By density of $D(A^\odot) = 
    (X^\odot)_1$, this estimate is also valid for any 
    $x^\odot \in X^\odot$.

As $X^\odot$ norms $X$, see, e.g., \cite[Thm.\ 1.3.5]{neervenAdjointSemigroupLinear1992}, with $m \coloneq \limsup_{t \searrow 0}\norm{T(t)}_{\LL(X)}$
it follows that 
\[\norm{x}_X \leq m \sup_{x^\odot \in \ball_{X^\odot}}\abs{\spr{x}{x^\odot}_{X,X'}}.\]
We combine the above estimates and get
\begin{align*}
    \norm{\Phi_1 u}_X &\leq m \sup_{x^\odot \in B_{X^\odot}} \abs{\spr{\Phi_1 u}{x^\odot}_{X\times X'}}\\
    &\leq 2 m \norm{u}_{\Ll_F([0,1];U)}\sup_{x^\odot \in B_{X^\odot}}\norm{B'|_{(X^\odot)_1}T^{\odot}(\,\cdot\,)x^\odot}_{\Ll_{\tilde F}([0,1];U')}\\
    &\leq 2 m M \norm{u}_{\Ll_F([0,1];U)}
\end{align*}
for any $u \in \Cc([0,1];U)$. Then, due to the density of $\Cc([0,1];U)$ in $\Ee_F([0,1];U)$, see, e.g., \cite[Lemma~1.2.28~(iii)]{hosfeldInputtostateStabilityClasses2025}, the inequality also extends to $u\in\Ee_F([0,1];U)$. Thus $B$ is $\Ee_F$-admissible.

For the converse, assume that $B$ is $\Ee_F$-admissible for
some Young function $F$.
Then there is $C\geq 0$ such that for any $u \in \Ee_F([0,1];U)$ we have
\[\Norm{\int_0^1 T(1-s) Bu(s) \dd s}_X \leq C \norm{u}_{\Ll_F([0,1];U)}.\]

Since $\norm{x}_X = \sup_{x'\in \ball_{X'}}\abs{\spr{x}{x'}_{X\times X'}} \geq 
\sup_{x^\odot\in \ball_{X^\odot}}\abs{\spr{x}{x^\odot}_{X\times X'}}$, the inequality
\[
    \sup_{x^\odot \in \mathrm{B}_{X^\odot}} \Abs{ \int_0^1 \spr{u(1-s)}{B'|_{(X^\odot)_1} T^\odot(s) x^\odot}_{U\times U'}\dd s} \leq
    C \norm{u}_{\Ll_F([0,1];U)}
\]
is valid for each $u \in \Ee_F([0,1];U)$. Defining $R_1 \in \LL(\Ll^1([0,1];U'))$ by
$f \mapsto f(1-\,\cdot\,)$, this implies that for every $x^\odot \in \ball_{X^\odot}$
\[f_{x^\odot}\colon u \mapsto \int_0^1 \spr{u(s)}{R_1B'|_{(X^\odot)_1} T^\odot(s) x^\odot}_{U\times U'}\dd s\]
is a bounded functional with norm $\norm{f_{x^\odot}}_{\Ee_F([0,1];U)'}$ uniformly bounded
in $x^\odot \in \mathrm{B}_{X^\odot}$ by $C$. 

Now we suppose $U'$ has the Radon--Nikodým property so that $\Ee_F([0,1];U) \simeq \Ll_{\tilde F}([0,1];U')$; see
\cite[Prop.\ 1.2.30]{hosfeldInputtostateStabilityClasses2025}. The representation
theorem yields a function $g_{x^\odot} \in \Ll_{\tilde F}([0,1];U')$
for each $x^\odot \in \ball_{X^\odot}$
such that $f_{x^\odot} = \int_0^1 \spr{\,\cdot\,}{g_{x^\odot}}_{U\times U'} \dd s$
and $\norm{g_{x^\odot}}_{\Ll_{\tilde F}([0,1];U')} \leq C$. Thus $g_{x^\odot} =  R_1 B'|_{(X^\odot)_1}T^\odot(\,\cdot\,)x^\odot$. Since $R_1$ is isometric with respect to $\norm{\,\cdot\,}_{\Ll_{\tilde F}}$, the set $\left.B'\right|_{(X^\odot)_1}T^\odot(\,\cdot\,)\ball_{X^\odot}$ is bounded in
$\Ll_{\tilde F}$, and therefore also uniformly integrable by de la Vallée Poussin's theorem.
\end{proof}
\begin{remark}
    This result can be compared to duality results for admissibility of control operators and
    observation operators. A classical result in this area
    was obtained by
    Weiss \cite[Thm.\ 6.9]{weissAdmissibleObservationOperators1989}, who formulated a duality theorem 
    for $\Ll^p$-admissible operators. Hosfeld \cite[Thm.\ 2.2.9]{hosfeldInputtostateStabilityClasses2025}, see also \cite[Thm.\ 5.2]{hosfeldCharacterizationOrliczAdmissibility2023}, 
    then generalized these results to the scale of Orlicz spaces. Note that
    these results make use of the structural assumption that the adjoint semigroup $T'$ is
        strongly continuous on $X'$. We also remark that the proof of Theorem \ref{thm:equi_for_B} above goes along
    some of the same lines as the proof of \cite[Thm.\ 2.2.9]{hosfeldInputtostateStabilityClasses2025}.

    For example, using Hosfeld's result \cite[Thm.\ 2.2.9]{hosfeldInputtostateStabilityClasses2025}, if 
    $T'$ is strongly continuous on $X'$ and $B$ is even admissible with respect to the larger space $\Ll_F([0,1];U)$, by way of Orlicz-boundedness it directly follows that $B'T'(\,\cdot\,)\ball_{X'}$
    is uniformly integrable, without needing to assume that $U'$ has the Radon--Nikodým property;
    in this case, the proof of ``a) $\Rightarrow$ b)'' from above is superfluous.

    As an aside on the Radon--Nikodým property: if $U'$ is separable or reflexive, then $U'$ has the Radon--Nikodým property by \cite[Thm.\ 1.3.21]{hytonenAnalysisBanachSpaces2016}.
    More precisely, it is known that $U'$ is Radon--Nikodým if and only if $U$ is an Asplund
    space, a property that is related to differentiability conditions on the norm of the Banach space $U$; see, e.g., \cite{diestelVectorMeasures1977,guiraoRenormingsBanachSpaces2022}.
\end{remark}
\subsection{Weak compactness}
As in Section \ref{sec:ui}, weak compactness is a useful tool to obtain uniform integrability. Given a $\Cc$-admissible $B\in \LL(U,X_{-1})$, such that 
the associated input operator $\Phi_1\colon \Cc([0,1];U) \to X$ is bounded, we show that weak compactness of
$\Phi_1$ suffices to show that $B$ is even Orlicz-admissible for some Young 
function.
The 
starting point of $\Ll^1$-boundedness of the formal dual operator $\tilde\Psi_1$ is provided by
\cite[Thm.\ 3.1]{aroraAdmissibleOperatorsSundual2025}. For later reference, we record
the following observation:
\begin{remark}\label{rem:isometryL1}
The embedding $\Ll^1([0,1];U') \hookrightarrow \Cc([0,1];U)'$ defined by \[f \mapsto T_f = \biggl(g \mapsto
\int_0^1 \spr{f}{g}_{U'\times U} \dd x\biggr)\] is an isometry. This can be seen using the Dinculeanu--Singer
representation of operators on spaces of Banach space-valued continuous functions, see \cite[182]{diestelVectorMeasures1977} and \cite[§19.3, Thm.\ 2]{dinculeanuVectorMeasures1967}. Due to the representation theorem, $\norm{T_f}$ is equal to the total variation of some representing measure. In this case, the representing
measure is the weighted Lebesgue measure $f \dd x$, with total variation (cf., e.g., \cite[Lem.\ 1.3.8]{hytonenAnalysisBanachSpaces2016}) given by $\norm{f}_{\Ll^1([0,1];U')}$. Thus
$f \mapsto T_f$ is an isometric embedding.
\end{remark}
\begin{theorem}\label{thm:mainwcompactness}
Let $X$ and $U$ be Banach spaces and assume $B$ is $\Cc$-admissible for a \czero-semigroup
$T$ on $X$.
Suppose that $\Phi_1\colon \Cc([0,1];U) \to X$, $u\mapsto\int_0^1 T(1-s)Bu(s)\dd s$ is weakly compact. Then there is a Young function $F$ such that $B$ is $\Ee_F$-admissible.
In particular, the operator $B$ is then also zero-class
    $\Cc$-admissible and zero-class
    $\Ll^\infty$-admissible.
\end{theorem}
\begin{proof}
As before, using the sun duality theorem \cite[Thm.\ 3.1]{aroraAdmissibleOperatorsSundual2025} we infer
 that 
 \[
    \tilde \Psi_1\colon X^\odot \to \Ll^1([0,1];U'),\qquad x^\odot \mapsto \left.B'\right|_{(X^\odot)_1}T^\odot(\,\cdot\,)x^\odot
\] 
defines a bounded linear map.
Due to Gantmacher's theorem, see \cite[Satz 1]{gantmacherUberSchwacheTotalstetige1940} and \cite[\,\!{3.5.13}]{megginsonIntroductionBanachSpace1998}, 
$\Phi_1' \colon X' \to \Cc([0,1];U)'$ is also weakly compact. 
Using $\Phi_1'$, we show that $\tilde \Psi_1$ is weakly compact. 

The proof is divided 
into two steps. First, we show that $\tilde\Psi_1$ is a shifted version 
of $\Phi_1'$ restricted to the sun dual space $X^\odot \subset X'$. Second,
we show that weak compactness of $\Phi_1'$ transfers to $\tilde \Psi_1$.

In the first step, we follow a strategy that is similar to the proof of \cite[Thm.\ 3.1]{aroraAdmissibleOperatorsSundual2025}.
Let $R_1 \in \LL(\Ll^1([0,1];U'))$ be the reflection operator $(R_1 f)(s) = f(1-s)$, $s \in [0,1]$.

We show that $\Phi_1'|_{X^\odot} = R_1 \tilde \Psi_1$.
Consider $u \in \Cc([0,1];U)$ and $x^\odot \in D(A^{\odot}) \subset X^\odot$.
We have
\[
    \spr{u}{\left.\Phi_1'\right|_{X^\odot} x^\odot}_{\Cc([0,1];U)\times\Cc([0,1];U)'} = \spr{\Phi_1 u}{x^\odot}_{X\times X'}
\]
by definition of the adjoint operator. Next, we expand the input operator $\Phi_1$ to see that 
\begin{align*}
    \spr{\Phi_1 u}{x^\odot}_{X \times X'} &= \Spr{\int_0^1 T(1-s) Bu(s) \dd s}{x^\odot}_{X\times X'}\\
    &= \int_0^1 \spr{T(1-s) Bu(s)}{x^\odot}_{X_{-1} \times (X_{-1})^\odot}\dd s\\
    &= \int_0^1 \spr{u(s)}{B'|_{(X^\odot)_1}T^\odot(1-s)x^\odot}_{U\times U'} \dd s\\
    &= \int_0^1 \spr{u(s)}{[{R}_1\tilde\Psi_1x^\odot](s)}_{U\times U'} \dd s.
\end{align*}
Once we employ the embedding of $\Ll^1([0,1];U')$ in $\mathrm C([0,1];U)'$, we conclude that for any $u \in \Cc([0,1];U)$ and any $x^\odot \in D(A^\odot)$
it holds that 
\[
    \spr{u}{\left.\Phi_1'\right|_{X^\odot} x^\odot}_{\Cc([0,1];U)\times\Cc([0,1];U)'} = \spr{u}{R_1 \tilde \Psi_1 x^\odot}_{\Cc([0,1];U)\times\Cc([0,1];U)'};
\]
density of $D(A^\odot)$ in $X^\odot$ allows us to infer that 
$\left.\Phi_1'\right|_{X^\odot} = R_1 \tilde \Psi_1$. In particular,
$\left.\Phi_1'\right|_{X^\odot}$ maps into $\Ll^1([0,1];U')$ seen as a subspace
of $\Cc([0,1];U)'$.

Now we show that $\tilde\Psi_1$ is weakly compact as an operator into
$\Ll^1([0,1];U')$. As a restriction to $X^\odot$ of a weakly compact operator on
$X'$, $\Phi_1'|_{X^\odot}$ is weakly compact into $\Cc([0,1];U)'$.
As noted above, the range of $\Phi_1'|_{X^\odot}$ is contained in $\Ll^1([0,1];U')$. Moreover, $\Ll^1([0,1];U')$, under the identification with the integral functionals with $\Ll^1$-densities, is  a \emph{closed} subspace of 
    $\Cc([0,1];U)'$. 
    
The closedness of the subspace $\Ll^1([0,1];U')$ follows from the Dinculeanu--Singer representation theorem, see Remark \ref{rem:isometryL1} above.
Since the $\Ll^1$ norm of a function $f$ agrees with the variation of the associated vector measure with density $f$, see, e.g., \cite[Lem.\ 1.3.8]{hytonenAnalysisBanachSpaces2016}, the representation theorem implies that 
    $\norm{f}_{\Ll^1([0,1];U')} = \norm{T_f}_{\Cc([0,1];U)'}$. Since $\Ll^1([0,1];U')$ is a Banach space, it follows
    that its isometric image is closed in $\Cc([0,1];U)'$.
    
    Since the weak topology of $\Ll^1([0,1];U')$ coincides with the weak topology induced by $\Cc([0,1];U)'$
    (cf, e.g., \cite[Prop.\ 2.5.22]{megginsonIntroductionBanachSpace1998}) and $\Ll^1([0,1];U')$ is even weakly
    closed in $\Cc([0,1];U)'$ by \cite[Cor.~V.1.5]{conwayCourseFunctionalAnalysis2007}, it follows that 
    $\Phi_1'|_{X^\odot} \in \LL(X^\odot, \Ll^1([0,1];U')$ is weakly compact. By the ideal property
    of weakly compact operators $R_1 \Phi_1'|_{X^\odot} = \tilde \Psi_1$ is also weakly compact; the Dunford--Pettis theorem implies that $\tilde\Psi_1(\ball_{X^\odot})$ is uniformly integrable.    
    
    By Theorem \ref{thm:equi_for_B}, it follows that there is a Young function 
    $F$ such that $B$ is $\Ee_F$-admissible. The other claims follow directly from 
    Orlicz-admissibility of $B$; see, e.g., \cite[Proof Prop.\ 5]{jacobContinuitySolutionsParabolic2019}.
\end{proof}
\begin{corollary}\label{cor:reflexiveforB}
    If $X$ is a reflexive Banach space, then any $\Cc$-admissible operator is $\Ee_F$-admissible for some 
    Young function $F$.
\end{corollary}
\begin{remark}
    It is known that $F\in \Delta_2$ can not be obtained from the above result in 
    general, even if $X$ is a separable Hilbert space. This was noted in \cite[Ex.~5.2, Rem.~4.3]{Jacob_2018}, where the authors note that requiring the $\Delta_2$ growth
    condition for the Young function $F$ would imply that the operator $B$ would need to be $\Ll^p$-admissible
    for some finite $p$. By constructing a counterexample, i.e., a $B$ operator, admissible with respect to $\Ll^\infty$ and not $\Ll^p$-admissible for any $p < \infty$, the authors show
    that this need not always be the case.
\end{remark}
\begin{remark}
    Question 27 in \cite{jacobContinuitySolutionsParabolic2019} on Orlicz-admissibility for
    analytic semigroups on Hilbert spaces $X$ and $\Ll^\infty$-admissible scalar operators $b \in \LL(\C,X_{-1})$ is resolved
    positively by Corollary \ref{cor:reflexiveforB} above.
\end{remark}
\subsection{Balancing assumptions on $X$ and $U$}
In this section we present a selection of conditions on the 
Banach spaces $X$ and $U$ that are sufficient to conclude that specific
classes of $\Cc$-admissible operators are all 
also $\Ee_F$-admissible for some Young function. In particular, we weaken the assumptions that $X$ is reflexive or of nontrivial type at the price of imposing extra conditions on the space $U$.
\subsubsection{Pełczyński's property (V)}\label{sec:prop(V)}
Using known results on Pełczyński's property (V), we present a set of conditions on Banach 
spaces $X$ and $U$ that imply that any operator $\Cc([0,1];U) \to X$ is weakly compact.
Property (V) traces back to Pełczyński's 
article \cite{pelczynskiBanachSpacesWhich1962}.
We recall the definition from \cite{Kalton1985}, see also \cite{Randrianantoanina1996,pelczynskiBanachSpacesWhich1962}. 

    Let $X$ and $Y$ be Banach spaces.
    An operator $T \in \LL(X,Y)$ is called unconditionally converging if $T$ maps each
    \wuc{} series $\sum_{n\in\N} x_n$ into a unconditionally $Y$-convergent series $\sum_{n\in \N} Tx_n$. 
    If for every Banach space $Y$ each unconditionally converging $T \in 
    \LL(X,Y)$ is also weakly compact, then $X$ is said to
    have (Pełczyński's) property (V).

    In \cite[Prop.\ 3]{pelczynskiBanachSpacesWhich1962}, it is shown that 
    if $X$ has (V) and $Y\not\supset c_0$, then any $T \in \LL(X,Y)$ is weakly compact.
    In our application, we will assume that $\Cc([0,1];U)$ has (V) and that $X \not\supset c_0$. 

    We summarize some known results 
    from the literature on Pełczyński's property (V).

    First, reflexive spaces have property (V), see, e.g., \cite[Prop.\ 7]{pelczynskiBanachSpacesWhich1962}. If $X$ has (V), then $X \not\supset c_0$
    and reflexivity of $X$ are equivalent by \cite[Prop.\ 8]{pelczynskiBanachSpacesWhich1962}.

    On property (V) for spaces of vector-valued continuous functions: if 
    $U$ is separable, then due to \cite[Thm.\ 1]{Randrianantoanina1996}, $\Cc([0,1];U)$ 
    has (V) if and only if $U$ does. Even if $U$ is non-separable, $U$ has (V) if 
    $\Cc(K;U)$ has (V); see, e.g., \cite[Proof\ Thm.\ 1]{Randrianantoanina1996},  \cite[Proof Thm.\ 6]{Kalton1985} and Lemma 
    \ref{lem:compl(V)} and Remark \ref{rem:Pelczynski_2} in the sequel. It follows 
    that $X$ must be reflexive if $X$ does not contain $c_0$ and $\Cc([0,1];X)$ has 
    property (V), which limits applicability of the argument if $U = X$.

    Lastly, by \cite[Thm.\ 9]{battLinearBoundedTransformations1969}, cf.\ also \cite{battIntegraldarstellungenLinearerTransformationen1967}, if $U$ is reflexive
and $X \not\supset c_0$, then any bounded operator $\Cc([0,1];U) \to X$ is is weakly compact, without
needing $U$ to be separable.
\begin{corollary}\label{cor:pelczynski_for_B}
    Let $X$ be a Banach space that does not contain a subspace isomorphic
    to $c_0$. Assume further that $U$ is a Banach space such that the
    space $\Cc([0,1];U)$ has Pełczyński's property (V). Then, if $T$
    is a \czero{}-semigroup and $B \in \LL(U,X_{-1})$ is $\Cc$-admissible
    for $T$, then there is a Young function $F$ such that $B$ is 
    $\Ee_F$-admissible for $T$.
\end{corollary}
\begin{proof}
    As $B$ is $\Cc$-admissible, the input operator $\Phi_1$ is 
    bounded $\Cc([0,1];U) \to X$. By the preceding remarks, $\Phi_1$ is weakly compact. The claim then follows from Theorem \ref{thm:mainwcompactness}.
\end{proof}
\subsubsection{The Grothendieck property}
A Banach space $X$ is called a Grothendieck space
if sequences in $X'$ converge weakly if and only if they converge weakly*, see, e.g., \cite[97]{diestel_grothendieck_1973}; for more
equivalent formulations of the Grothendieck property, see also \cite[Ch.\ Five, §2, Cor.\ 5]{diestelGeometryBanachSpaces1975}. Compare also the survey \cite{gonzalez_grothendieck_2021}.

Some examples for Grothendieck spaces are reflexive spaces \cite[97]{diestel_grothendieck_1973},
spaces of continuous functions $\Cc(K)$ for compact sets $K$ satisfying certain disjointness
assumptions for closures in $K$ (e.g., Stonean and $\sigma$-Stonean spaces $K$) \cite[269, 271]{gonzalez_grothendieck_2021}. By identifying $\Ll^\infty(\Omega,\mu)$ with a suitable $\Cc(K)$ space,
it follows that $\Ll^\infty(\Omega, \mu)$ is Grothendieck, see \cite[Prop.\ 4.1.5]{gonzalez_grothendieck_2021}.

If $X$ is Grothendieck and $Y$ is a \emph{weakly compactly
generated} Banach space, then any $S\in\LL(X,Y)$ is weakly compact, see, e.g., \cite[Ch.\ Five, §2, Cor.\ 5]{diestelGeometryBanachSpaces1975}.
Note that, by definition, a space $Y$ is weakly compactly generated (\WCG{}) if there exists a weakly compact subset of $Y$ with dense span, see, e.g., \cite[143]{diestelGeometryBanachSpaces1975}. From \cite[143]{diestelGeometryBanachSpaces1975}, some examples for \WCG{} spaces are 
reflexive spaces, separable spaces, $\Ll^1(\mu)$ for $\sigma$-finite measures $\mu$ and the space $c_0$.
We also refer to \cite{guiraoRenormingsBanachSpaces2022} for more on \WCG{} Banach spaces.
\begin{theorem}\label{thm:grothendieck}
    Let $U$ be a Banach space such that $\Ll^\infty([0,1];U)$ has the Grothendieck
    property. Assume that $X$ is \WCG{}
    and let $T$ be a \czero{}-semigroup on $X$. Then any $\Ll^\infty$-admissible
    $B \in \LL(U,X_{-1})$ is also $\Ee_F$-admissible for some
    Young function $F$. In particular, $B$ is then also zero-class $\Ll^\infty$-admissible for $T$.
\end{theorem}
\begin{proof}
     Since $\Phi_1 \in \LL(\Ll^\infty([0,1];U),X)$ is weakly compact, this follows
     from Theorem \ref{thm:mainwcompactness}.
     \end{proof}
\begin{remark}
    Assuming that $\Ll^\infty([0,1];U)$ has the Grothendieck property places severe restrictions on the class of possible Banach spaces $U$. In particular, by \cite[Prop.\ 5.4.5 (2)]{gonzalez_grothendieck_2021}, if $\mu$ is 
    non-atomic and $\Ll^\infty(\mu, U)$
    is a Grothendieck space, then $U$ not only has nontrivial type, but $U$ is 
    also reflexive. So even in the case that $\Ll^\infty([0,1];X)$ was
    a Grothendieck space for some $X$, Orlicz-admissibility for any $\Cc$-admissible $B$ (and thus in particular for any $\Ll^\infty$-admissible $B$) already 
    follows from the necessary reflexivity of the space $X$, which rules 
    out an application in the maximal regularity context with $X=U$. For this reason, the main application of this Grothendieck criterion is in the situation $U = \C$.
\end{remark}
\begin{remark}
If the space $X$ is \emph{sun-reflexive} with respect to the semigroup $T$, more can be said.
We cite the definition from \cite[7]{neervenAdjointSemigroupLinear1992}: The Banach space $X$ is sun-reflexive or $\odot$-reflexive if $j\colon X \to X^{\odot\prime}$, defined by $\spr{jx}{x^\odot}_{X^{\odot\prime},X^\odot}
\coloneq \spr{x^\odot}{x}_{X',X}$ for $x^\odot \in X^\odot$, surjects onto $X^{\odot\odot}$,
the space of strong continuity of the semigroup $T^{\odot\prime}$.

By \cite[37]{neervenAdjointSemigroupLinear1992}, if there is a \czero-semigroup
such that $X$ is sun-reflexive, then $X$ is \WCG. Hence, if $b \in \LL(\C,X_{-1})$
is $\Ll^\infty$-admissible (for some fixed generator $A$ on $X$; $A$ need not generate 
the semigroup that makes $X$ sun-reflexive), then $b$ is also $\Ee_F$-admissible
for some Young function $F$. This follows from Theorem \ref{thm:grothendieck} since
$\Ll^\infty([0,1])$ is Grothendieck.
\end{remark}
\begin{remark}
    If we choose $X = c_0$ and $U=\C$, given any semigroup $T$ on $c_0$, any $\Ll^\infty$-admissible 
    operator $b\in \LL(\C,(c_0)_{-1})$ must already be $\Ee_F$-admissible for some Young function $F$ due 
    to Theorem \ref{thm:grothendieck} as 
    $c_0$ is \WCG. In particular, $X=c_0$ and $U = \C$ is an example where $\Cc([0,1];U)$ has
    property (V) but Corollary \ref{cor:pelczynski_for_B} does not apply.
\end{remark}
\begin{remark}
    Question 28 in \cite{jacobContinuitySolutionsParabolic2019} on Orlicz-admissibility for $X=\Ll^p(\Omega,\mu)$ and any $b\in X_{-1}$ when $A$ has a bounded $\mathrm{H}^\infty$-functional calculus has a positive answer when $p > 1$ since $X$ is reflexive; when $p = 1$, we can also conclude Orlicz-admissibility if $\mu$ is $\sigma$-finite or $\mu$ is such that $\Ll^1(\Omega,\mu)$ is separable; since $\Phi_1\colon \Ll^\infty([0,1]) \to \Ll^1(\Omega,\mu)$ is then an operator from a Grothendieck space into a \WCG{} space, $\Phi_1$ is then also weakly compact.
\end{remark}
\subsubsection{Type assumptions}
Nontrivial type of $X$ can also be used to conclude Orlicz-admissibility:
\begin{theorem}\label{thm:nikishinforB}
    If $X$ has nontrivial type, then given a
    \czero{}-semigroup $T$ on $X$, and a Banach space $U$, every
    $\Cc$-admissible
    $B \in \LL(U,X_{-1})$ is $\Ee_F$-admissible for some Young function $F$.
\end{theorem}
\begin{proof}
    If $X$ has nontrivial type, then $X'$ has nontrivial type, see, e.g., \cite[Cor.\ 7.4.24]{hytonenAnalysisBanachSpaces2018}. The subspace $X^\odot$ of $X'$ then also
    has nontrivial type; see, e.g., \cite[Rem.\ 6.2.11\ (f)]{albiacTopicsBanachSpace2016}. By the sun duality theorem \cite[Thm 3.1]{aroraAdmissibleOperatorsSundual2025}, the map $x^\odot \mapsto B'|_{(X^\odot)_1}T^\odot(\,\cdot\,)x^\odot$ extends to a bounded map $X^\odot \to \Ll^1([0,1];U')$. As in the proof of Corollary
\ref{cor:typeforCs}, $X^\odot$ does not contain $\ell^1$. Due to
Proposition \ref{prop:complementedforui}, it follows that
$B'|_{(X^\odot)_1}T^\odot(\,\cdot\,)\mathrm{B}_{X^\odot}$ is 
uniformly integrable. By Theorem \ref{thm:equi_for_B}, $B$ is $\Ee_F$-admissible for some Young function $F$.
\end{proof}
\begin{remark}
It is open if the statement regarding Orlicz-admissibility from above can be proven under the assumption that $X$ has finite cotype. Note that $\operatorname{cotype}(X) < \infty$
does not imply $\operatorname{type}(X') > 1$; a counterexample
is $X = \ell^1$, see, e.g. \cite[Cor.\ 7.1.10, Sect.\ 7.1.c]{hytonenAnalysisBanachSpaces2018}.

Regarding factorization theorems: Pisier's factorization theorem, see, e.g, 
\cite[Thm.\ 2.4, Cor.\ 2.7]{pisierFactorizationOperatorsL_pinfty1986} and \cite[Thm.\ 7.2.4]{hytonenAnalysisBanachSpaces2018}, concerns
operators $\Cc(K) \to X$, where $K$ 
is a compact Hausdorff space and $X$ has finite cotype.
If $X$ has finite cotype and $U=\C$, the theorem could be used to conclude $\Ee_F$-admissibility 
for $\Cc$-admissible $B$.
However, $\Cc(K)$ has property (V), see, e.g., \cite[Thm.\ 1]{pelczynskiBanachSpacesWhich1962}; as the argument with property (V) does not need any cotype
assumption for $X$, only the weaker assumption that $X\not\supset c_0$, Pisier's theorem is not suitable for our application.
\end{remark}
\begin{remark}
    Question 24 in \cite{jacobContinuitySolutionsParabolic2019} on Orlicz-admissibility for $\Ll^\infty$-admissible operators is resolved positively if any of the preceding sufficient conditions of Section \ref{sec:B_operators} is met.
\end{remark}
\section{Applications}
\label{sec:applications}
\subsection{Admissibility for observation operators}\label{sec:admissibilityobservation}
In systems theory, the notion of admissibility is also well-studied in the context 
of \emph{observation operators}, where the notion is a pivotal part of well-posedness of 
linear systems. 
The general setting of ``semigroup control systems'', supporting the modeling of unbounded observation and control operators on Hilbert spaces, was introduced by Salamon in the fundamental paper  \cite{salamonInfiniteDimensionalLinear1987}. There, the notion of admissibility (with
respect to $\mathrm{Z} = \Ll^2$) of control operators appears as hypothesis (S2) and 
$\Ll^2$-admissibility of observation operators appears as hypothesis (S3).
Later, Weiss, in two seminal works \cite{weissAdmissibleObservationOperators1989,weiss_admissibility_1989} on admissible control and observation operators laid the foundation for the study of the property of admissibility on Banach spaces, providing many general results in the field. For example, an important result is
\cite[Thm.\ 6.9]{weissAdmissibleObservationOperators1989} on duality relations between $\Ll^p$-admissible observation and control operators.
In this context we also mention the comprehensive treatise \cite{staffans_well-posed_2005} on well-posed systems by Staffans; among the many results on admissibility contained therein, we highlight 
the treatment of admissibility with respect to the space of regulated functions, cf.\ \cite[Thm.\ 10.2.2]{staffans_well-posed_2005}. Moreover, the author also considers many special cases,
such as diagonal, normal and contractive semigroups on Hilbert spaces \cite[Sects.\ 10.6, 10.7]{staffans_well-posed_2005}.
See also the survey article \cite{jacobAdmissibilityControlObservation2004} by Jacob and Partington
for an overview of results on admissibility theory.

The present section will be devoted to the discussion of the consequences of Section \ref{sec:operatorsintoL1} for admissibility of observation operators. Let $W$ be a Banach space and $A$ generate a semigroup $T$ on $W$. An observation operator is
an operator $C\in \LL(W_1,Y)$; $C$ need not be bounded $W \to Y$.

Consider the following system:
\begin{align*}
    \dot x(t) &= Ax(t), \qquad t\geq 0,\ \\
    x(0) &= x_0 \in X,\\ y(t) &= Cx(t), \qquad t\geq 0.
\end{align*}
We can ask if the map $w \mapsto CT(\,\cdot\,)w$,
initially defined on $D(A) = W_1$, extends to a bounded operator 
$\Psi_1\colon W \to \mathrm{Z}([0,1];Y)$. 
If $\Psi_1$ does extend to a map in $\LL(W,\mathrm{Z}([0,1];Y))$, we say that $C$ is a $\mathrm{Z}$-admissible observation operator for $T$ (or for
$A$, the generator of $T$). If it is clear from the context that a given operator $C$ is to 
be understood as an observation operator, then we say that $C$ is $\mathrm{Z}$-admissible.
Also, the operator $C$ is called zero-class $\mathrm{Z}$-admissible
if 
\[
    \sup_{w\in \ball_W} \norm{(\Psi_1w)|_{[0,\epsilon]}}_{\mathrm{Z}([0,\epsilon];Y)}
    \to 0\text{ as }\epsilon\searrow 0.
\]
In the sequel, the minimal assumption that we make is that $C$ is $\Ll^1$-admissible for $T$, ensuring that the image $\Psi_1(\ball_W)$ is a subset of $\Ll^1([0,1];Y)$. Given
suitable Banach spaces $W$, the results of Section
\ref{sec:operatorsintoL1} show that any $\Ll^1$-admissible $C$ is even $\Ee_F$-admissible 
for some Young function $F\in\Delta_2$. More precisely, we have the following theorem:
\begin{theorem}\label{thm:OrliczforCs}
    Let $W$ and $Y$ be Banach spaces and let $T$ be a \czero-semigroup on $X$. Assume that $C\in \LL(W_1,Y)$ is a $\Ll^1$-admissible observation operator for $T$.
    In any of the following cases, there exists a Young function $F \in \Delta_2$ such that
    $C$ is also $\Ee_F$-admissible for $T$:
    \begin{enumerate}[label=\alph*\textup{)}]
        \item $W$ is reflexive;\label{item:Wreflexive}
        \item $W$ has nontrivial (Rademacher) type;\label{item:Nontrivialtype}
        \item $W$ does not contain a complemented subspace isomorphic to $\ell^1$;\label{item:Nocomplementedell1}
        \item $\Psi_1$ is a weakly compact operator;\label{item:WeakCompactPsi}
        \item $\Psi_1(\ball_W)$ is a uniformly integrable subset of $\Ll^1([0,1];Y)$.\label{item:UIPsiBallW}
    \end{enumerate} 
    In particular, in this case, $C$ is zero-class $\Ll^{1}$-admissible.
\end{theorem}
\begin{proof}
    By Theorem \ref{thm:ui_C}, \ref{item:UIPsiBallW} is equivalent to existence of a Young function $F \in \Delta_2$
    such that $\Psi_1$ factors through $\Ee_F$. In turn, this is equivalent
    to the assertion that $C$ is $\Ee_F$-admissible. Next, 
    \ref{item:WeakCompactPsi} implies \ref{item:UIPsiBallW} due to the 
    Dunford--Pettis theorem. Item \ref{item:Wreflexive}, reflexivity of $W$, implies that any bounded operator
    defined on $W$ is weakly compact, such that \ref{item:WeakCompactPsi} holds. See also Proposition \ref{prop:reflexiveui} above. By Proposition \ref{prop:complementedforui},
    uniform
    integrability of $\Psi_1(\ball_W)$ follows if $W$ does not 
    contain $\ell^1$ complementedly, i.e., \ref{item:Nocomplementedell1} implies \ref{item:UIPsiBallW}. Lastly,
    if $W$ has nontrivial type, then $W$ cannot contain $\ell^1$, such that \ref{item:Nocomplementedell1} follows
    from \ref{item:Nontrivialtype}; see Corollary \ref{cor:typeforCs} above.
\end{proof}
\begin{remark}
Results on improving $\Ll^1$-admissibility of observation operators based on weak compactness properties
have previously appeared in \cite[Thm.\ 3.9, Rem.\ 3.10]{aroraLimitcaseAdmissibilityPositive2025a}. There, the authors show that \emph{zero-class} $\Ll^1$ admissibility follows when the operator $\Psi_1$ is weakly compact. However, the argument
there uses a version of the Dunford--Pettis theorem that
requires $Y$ is an AL-space.
\end{remark}
Uniform integrability of $\Psi_1(\ball_W)$
directly yields zero-class $\Ll^1$-admissibility of the associated $C$ operator.
Conversely, even when $Y=\C$, a $\Ll^1$-admissible
operator need not be zero-class $\Ll^1$-admissible:
\begin{example}\label{ex:nonzeroclassC}
Consider the Banach space $W = \ell^1$ and the operator $A = \operatorname{diag}(-n)$ with domain
$D(A) = \{ w \in W :  \sum_{n=1}^\infty \abs{nw_n} < \infty \}$. Then $A$
generates a \czero{}-semigroup $T$ on $W$ given by
$(T(t)w)_n = \e^{-nt}w_n$ for $n\in \N$ and $w\in W$. See also \cite[Ex.\ 17.4.6]{hytonenAnalysisBanachSpaces2023} where this semigroup is discussed in the context of maximal $\Ll^1$- and $\Ll^\infty$-regularity.

We define the observation operator $C \in \LL(D(A), \C)$ by
$Cw = \sum_{n=1}^\infty nw_n$ for $w \in D(A)$.
For $w \in D(A)$, set $(\Psi_1 w)(t) \coloneq CT(t)w = \sum_{n=1}^\infty n\e^{-nt}w_n$, $t\in [0,1]$.
For an arbitrary $\epsilon > 0$, we denote the truncation of $\Psi_1 w$ to $[0,\epsilon]$
by $\Psi_1^{(\epsilon)}$. We get:
\begin{align*}
    \norm{\Psi_1^{(\epsilon)} w}_{\Ll^1([0,\epsilon])} &= 
    \int_0^\epsilon \Biggl|\sum_{n=1}^\infty n\e^{-nt}w_n\Biggr|\dd t \leq \int_0^\epsilon \sum_{n=1}^\infty n\e^{-nt} \abs{w_n} \dd t\\
    &= \sum_{n=1}^\infty \abs{w_n} \int_0^\epsilon n\e^{-nt} \dd t = \sum_{n=1}^\infty \abs{w_n} ( 1 - \e^{-n\epsilon})
    \leq \norm{w}_{\ell^1}.
\end{align*}
It follows that $C$ is $\Ll^1$-admissible and $\Psi_1^{(\epsilon)}$ is even bounded on $W = \ell^1$ with $\norm{\Psi_1^{(\epsilon)}}_{\LL(\ell^1,\Ll^1([0,\epsilon]))}\leq 1$.

To show that $C$ is not zero-class admissible, we let
$e_k$ be the $k$th element of the standard basis of $\ell^1$;
we obtain $\norm{\Psi^{(\epsilon)}_1 e_k}_{\Ll^1([0,\epsilon])} = 
    \int_0^1 k\e^{-kt} \dd t = (1 - \e^{-k\epsilon})$,
which implies that $\norm{\Psi^{(\epsilon)}_1}_{\LL(\ell^1,\Ll^1([0,\epsilon]))} \geq (1 - \e^{-k\epsilon})$ for each $k \in \N$. Therefore, we must have
\[
    1 = \lim_{k\to \infty} (1 - \e^{-k\epsilon}) \leq \norm{\Psi_1^{(\epsilon)}}_{\LL(\ell^1,\Ll^1([0,\epsilon]))}\leq 1
\]
which means that $\norm{\Psi^{(\epsilon)}_1} = 1$ for every $\epsilon > 0$.
In particular, uniform integrability fails on the shrinking intervals $[0,\epsilon]$ as $\epsilon \to 0$. 
\end{example}
Another $\Psi_1$ operator that fails the zero-class property is the identity on $\Ll^1([0,1])$; this
operator can be constructed from a semigroup and an observation operator, see \cite[Ex.\ 26]{jacobContinuitySolutionsParabolic2019}, where the authors consider this example in the context of admissibility and Orlicz spaces.
\begin{remark}
    Question 25 in \cite{jacobContinuitySolutionsParabolic2019} on Orlicz-admissibility
    for $\Ll^1$-admissible observation operators, which the authors show to have a negative answer in general \cite[Ex.\ 26]{jacobContinuitySolutionsParabolic2019}, can be resolved positively in any of the cases presented in Theorem \ref{thm:OrliczforCs} above.
\end{remark}
\subsection{Input-to-state stability 
and integral input-to-state stability}
The Orlicz-admissibility result of Theorem~\ref{thm:mainwcompactness} above has implications in the theory of input-to-state stability (\ISS); a property first investigated by Sontag in \cite{sontagSmoothStabilization1989} for finite-dimensional nonlinear control systems. For an overview of the
theory of \ISS, we refer to the survey  \cite{mironchenkoInputtoStateStabilityInfiniteDimensional2020}.

Let $X$ be a Banach space and let $A$ generate a \czero-semigroup on $X$. For
a given input Banach space $U$ and an operator 
$B\in\LL(U,X_{-1})$, we consider the linear system
associated to the pair $(A,B)$, which is the abstract differential equation
\[
    \dot x(t) = Ax(t) + Bu(t),\qquad t\geq 0.
\]
For a given initial value $x_0 \in X$, the mild solution is 
$x(t) = T(t)x_0 + \int_0^t T(t-s) Bu(s) \dd s$; for a general operator $B\in\LL(U,X_{-1})$,
the function
$x$ merely exists as a (continuous) function with values in the extrapolation space $X_{-1}$.
As remarked above, $\mathrm{Z}$-admissibility guarantees that solutions take values in $X$
for input functions $u \in \mathrm{Z}([0,1];U)$.

Input-to-state stability of a control system with respect to some function space $\mathrm{Z} = 
\mathrm{Z}([0,\infty);U)$ 
is a notion that guarantees 
that the system is, in a sense, well-behaved for any input $u \in \mathrm{Z}$ by bounding the mild solution by functions of the initial value $x_0\in X$ and the input $u\in \mathrm{Z}$. 
To give a complete definition of input-to-state stability, we recall the definition of
the function classes $\mathcal{K}_\infty$ and $\mathcal{K}\mathcal{L}$ from 
\cite[537]{mironchenkoInputtoStateStabilityInfiniteDimensional2020}:
\begin{itemize}
    \item A function $f\colon [0,\infty) \to [0,\infty)$ lies in $\mathcal{K}_\infty$
    if: $f$ is continuous, zero only at zero, strictly increasing and unbounded;
    \item A function $g\colon [0,\infty) \times [0,\infty) \to [0,\infty)$
    lies in $\mathcal{K}\mathcal{L}$ if: $g$ is continuous and for any $t\geq 0$, $g(\,\cdot\,, t)$ is zero only at zero and strictly 
        increasing, and for any $s > 0$, $g(s,\,\cdot\,)$ is strictly decreasing and 
        $g(s,t) \to 0$ as $t \to \infty$.
\end{itemize}
For \ISS, we cite 
the following definition:
\begin{definition}[{\cite[p.~873, Definition~2.7]{Jacob_2018}}]
    Let $X$ and $U$ be Banach spaces and let $A$ geneate a \czero-semigroup
    on $X$. Consider an operator $B\in\LL(U,X_{-1})$.
    \begin{itemize}
    \item If there are 
    functions $\beta \in \mathcal{KL}$ and $\mu \in \mathcal{K}_\infty$ such that the mild solution $x$ of the system $(A,B)$ is $X$-valued and satisfies the inequality
    \[
        \norm{x(t)}_X \leq \beta(\norm{x_0}_X,t)+\mu(\norm{u}_{\mathrm{Z}(0,t;U)}),\qquad t\geq0,
    \]
    for any input function $u \in \mathrm{Z}(0,\infty;U)$ and any initial 
    value $x_0 \in X$, then $\Sigma(A,B)$ is
    called $\mathrm{Z}$-input-to-state stable; or, in brief, $\mathrm{Z}$-\ISS.

    \item 
    The system $(A,B)$ is called integral input-to-state stable with respect to $\mathrm{Z}$ or $\mathrm{Z}$-\iISS{} if the 
    mild solution $x$ takes values in $X$ and there are functions $\beta \in \mathcal{KL}$, $\mu \in \mathcal{K}$ and $\theta\in\mathcal K_\infty$
    such that for each $u \in \mathrm{Z}(0,\infty;U)$
    \[
        \norm{x(t)}_X \leq \beta(\norm{x_0}_X, t) + \theta\biggl( \int_0^t \mu(\norm{u(s)}_U)\dd s \biggr)
    \]
    holds for all $t \geq 0$.
    \end{itemize}
\end{definition}
It is a simple fact that  $\mathrm{Z}$-\iISS{} always implies $\mathrm{Z}$-\ISS, \cite[Prop.\ 2.10]{Jacob_2018} and that the latter property is equivalent to $\mathrm Z$-admissibility of $B$ and exponential stability of the semigroup. 
For finite-dimensional linear systems, it is well-known that \ISS{} and \iISS{}
are equivalent; see, e.g., \cite{sontagCommentsIntegralVariants1998}. In the 
case of $\Cc$-\ISS{} and $\Cc$-\iISS{}, the two properties are not equivalent. 
This follows from the Kato example in \cite[Ex.\ 2.3, Rem. 2.4]{jacobRefinementBaillonTheorem2022}
of an exponentially stable semigroup on $X = c_0$ together with an operator $B$
that is $\Cc$-admissible but not zero-class $\Cc$-admissible. The $\Cc$-\iISS{} property
fails due to \cite[Thm.\ 3.1]{Jacob_2018} since $(A,B)$ is not zero-class $\Cc$-admissible. 

The paper \cite{Jacob_2018} gives conditions for \iISS{} based on Orlicz-admissibility 
for control systems on infinite-dimensional state spaces. There 
it is shown that, under the assumption of exponential stability of the
associated semigroup, admissibility implies \ISS{}; and that $\Ee_F$-\ISS{} 
for some Young function $F$ implies $\Ll^\infty$-\iISS{}. We also refer to \cite{arora2026inputtostatestabilityintegralnorms} for extensions of
the \ISS{} concept to other norms.

With our previous
results,
for certain classes of Banach spaces and control systems, we 
resolve the question if $\Ll^\infty$-\ISS{} implies $\Ll^\infty$-\iISS{}, which is open in general; see, e.g. \cite{Jacob_2018}.
\begin{theorem}\label{thm:ISS}
    Let $A$ generate a \czero{}-semigroup $T$ on a Banach space $X$. Let $U$
    be a further Banach space and $B \in \LL(U,X_{-1})$ be a control operator.  
    If $(A,B)$ is $\Cc$-\ISS{} and the corresponding input operator $\Phi_1\colon\Cc([0,1];U)\to X$ is weakly compact, then the system is $\Ll^\infty$-\iISS{}.
\end{theorem}
\begin{proof}
    Since $(A,B)$ is $\Cc$-\ISS{}, by \cite[Prop.~2.10]{Jacob_2018}, $B$ is also $\Cc$-admissible and the semigroup generated by $A$ is exponentially stable.
    Using weak compactness of $\Phi_1$, by Theorem~\ref{thm:mainwcompactness}, there exists a Young function $F$ such that 
    $B$ is also $\Ee_F$-admissible for $A$. As $T$ is exponentially stable, \cite[Prop.~2.10]{Jacob_2018} implies that the system is $\Ee_F$-\ISS{}. Using \cite[Thm.~3.1]{Jacob_2018}, it follows that the system is even $\Ll^\infty$-\iISS{}.
\end{proof}
We get the following from Theorems \ref{thm:ISS}, \ref{thm:nikishinforB}, \ref{cor:pelczynski_for_B} and \ref{thm:grothendieck}.
\begin{corollary}\label{cor:iISSandISS}
    Let $X$ be a Banach space, $A$ generate a \czero-semigroup on 
    $X$ and $B$ be a control operator $U \to X_{-1}$ for another Banach space $U$. If either
    \begin{itemize}
        \item $X$ is a reflexive Banach space or of nontrivial type, or
        \item $X$ does not contain $c_{0}$ and $U$ is finite-dimensional,
    \end{itemize}
    then the following assertions are equivalent 
    \begin{enumerate}[label=\alph*\textup{)}]
        \item $(A,B)$ is $\Cc$-\ISS{};
        \item $(A,B)$ is  $\Ll^\infty$-\ISS{};
        \item $(A,B)$ is  $\Cc$-\iISS{};
        \item $(A,B)$ is  $\Ll^\infty$-\iISS{}. 
    \end{enumerate}
    Furthermore, assertions b)--d) are also equivalent if $X$ is \WCG{} and $U$ is finite-dimensional.
\end{corollary}
This partly generalizes \cite[Cor.\ 21]{jacobContinuitySolutionsParabolic2019} as the coincidence of
\iISS{} and \ISS{} there is derived under the assuption that $X$ is
a Hilbert space, among other assumptions on $A$ and $U$. Moreover, the above result also generalizes \cite[Cor.\ 3.7]{jacob2026laplacecarlesonembeddingsinfinitynormadmissibility}, as a space $X$ with a $q$-Riesz basis must be separable and thus \WCG.
We remark that a proof or disproof of the general equivalence of $\Ll^\infty$-\ISS{} and $\Ll^\infty$-\iISS{} (cf.\ \cite{Jacob_2018}) in the cases not covered by Corollary \ref{cor:iISSandISS} above remains open.
\subsection{Perturbation theory}
The Desch--Schappacher and Miyadera--Voigt theorems, see, e.g., \cite[Chap.\ III]{engel_one-parameter_2000},
lead to the above admissibility results having consequences for generation of perturbed semigroups.
These results give conditions on semigroup generation for additively perturbed semigroup
generators. Some of the many recent related works are \cite{aroraLimitcaseAdmissibilityPositive2025a, barbieriPerturbationsPositiveSemigroups2025, barbieriOnStructuredPerturbations2025, batkaiPerturbationsPositiveSemigroups2018, wintermayrPositivityPerturbationTheory2019a} where the relation between positivity properties
and perturbations are studied. See also \cite{kunstmannPerturbationTheorems2001} for perturbation
results in the context of maximal regularity. We also remark that these perturbation results
are particularly useful for \emph{boundary perturbations}; that is, when the operator $B$
arises from a perturbation of the boundary conditions of a differential operator $A$. In this
case, the operator $B$ is indeed of the form $B = A_{-1}B_0$. See, e.g., \cite{greinerPerturbingtheBoundaryConditions1987} and \cite[Chap.\ 10]{tucsnak_observation_2009}.

We cite the Desch--Schappacher theorem in its version from \cite[Corollary III.3.3]{engel_one-parameter_2000}, see also \cite[Theorem III.3.1]{engel_one-parameter_2000}. The theorem 
states the following: Let $T$ be 
a \czero-semigroup on a Banach space $X$, generated by $A$. Assume $B \in \LL(X,X_{-1})$ is such that 
there are $t_0 > 0$ and $0 \leq \kappa(t_0) < 1$ with the property that 
\[
    \Norm{\int_0^{t_0}T(t_0 - s)Bu(s)\dd s}_X \leq \kappa(t_0) \norm{u}_{\Cc([0,t_0];X)}
\]
holds for each continuous $u\colon [0,t_0]\to X$. Then the operator
$(A_{-1} +B)_X$ with domain $D((A_{-1} +B)_X) = \{ x \in  X : A_{-1}x + Bx \in X \}$
generates a strongly continuous semigroup on $X$.

The requirements from above can be interpreted as asking $B$ to be 
admissible with respect to the continuous functions, and additionally that
the constant can be chosen to be less than one for some $t_0$. Hence the ``smallness condition'' $\kappa(t_0) < 1$ is always 
satisfied if $B$ is even zero-class $\Cc$-admissible.
\begin{theorem}
    Let $X$ be a Banach space that is reflexive or has nontrivial type. If
    $A$ generates a \czero-semigroup on $X$ and $B\in\LL(X,X_{-1})$ has the
    property that $\int_0^1 T(1-s)Bu(s) \dd s \in X$ for each $u\in \Cc([0,1];X)$,
    then $(A_{-1} + B)_X$ also generates a \czero-semigroup on $X$.
\end{theorem}
\begin{proof}
    The assumption on $B$ implies that $B$ is $\Cc$-admissible. Using Corollary \ref{cor:reflexiveforB}
    for reflexive $X$ or Theorem \ref{thm:nikishinforB} for $X$ with nontrivial 
    type, it follows that there is a Young function $F$ such that 
    $B$ is also $\Ee_F$-admissible for $A$. By \cite[Proof of Prop.\ 5]{jacobContinuitySolutionsParabolic2019}, $B$ is zero-class $\Cc$-admissible, i.e.,
    \[
        \Norm{ \int_0^{t} T(t - s) Bu(s) \dd s}_X \leq \kappa(t) \norm{u}_{\Cc([0,t];X)}
    \]
    holds for each $t > 0$ and any $u \in \Cc([0,t];X)$ and we additionally have $\kappa(t) \to 0$ for $t \searrow 0$.
    Due to \cite[Cor.\ III.3.3]{engel_one-parameter_2000}, this implies that $B$ is a valid Desch-Schappacher perturbation for $A$. Hence $(A_{-1} + B)_X$ generates
    a strongly continuous semigroup on $X$.
\end{proof}
The corresponding version for observation operators $C$ is formulated using the 
Miyadera--Voigt perturbation theorem, which can be found in, e.g., \cite[Theorem III.3.14, Corollary III.3.16]{engel_one-parameter_2000}. The perturbation theorem
states that if $A$ generates a strongly continuous semigroup on $X$ and $C \in\LL(X_1, X)$ is $\Ll^1$-admissible with admissibility constant strictly
less than one, then the operator $A + C$ with domain $D(A + C) = D(A) = X_1$ is itself the generator of a strongly continuous semigroup on $X$. Hence:
\begin{corollary}
    If $X$ is reflexive or has nontrivial type, then any $\Ll^1$-admissible
    observation operator $C \in \LL(X_1,X)$ is a Miyadera--Voigt perturbation, such
    that $A + C$ is the generator of a strongly continuous semigroup on $X$.
\end{corollary}
\subsection{Continuity of mild solutions}
For linear systems with admissible operators a natural question
is that of continuity of the mild solutions. Let $X$ be a Banach space,
$T$ be a \czero-semigroup on $X$, $U$ a second Banach space and $B\in \LL(U,X_{-1})$
be an operator. 

Suppose $B$ is $\Ll^p$-admissible for some $1 \leq p < \infty$. By a well-known 
result of Weiss \cite[Prop.\ 2.3, Thm.\ 3.9]{weiss_admissibility_1989} (see also \cite[Rem.\ 2.6]{Jacob_2018}) 
the mild solutions of the system
\[
\begin{split}
    \dot x(t) &= Ax(t) + Bu(t),\qquad t\geq 0,\\ 
    x(0) &= x_0 \in X 
\end{split}\tag{$\Sigma$}\label{eq:sys}
\]
are functions $x\in \Cc([0,t_0];X)$ for any given $u \in \Ll^p([0,\infty);U)$
and any initial datum $x_0 \in X$. In \cite[Prob.\ 2.4]{weiss_admissibility_1989}
Weiss asks if this property extends to $\Ll^\infty$-admissible $B$.

Building on Section \ref{sec:B_operators}, 
we answer Weiss' question affirmatively for a range of special cases;
namely those where Orlicz-admissibility follows automatically
from $\Ll^\infty$- or $\Cc$-admissibility.
Due to Prop.\ 2.1.9 in \cite{hosfeldInputtostateStabilityClasses2025},
the mild solution of any $\Ee_F$-admissible system with 
is continuous for any input $u\in \Ee_F([0,\infty);U)$. Hence we have:
\begin{corollary}\label{cor:linftymildsol}
    Let $B$ be $\Ll^\infty$-admissible. If 
    \begin{itemize}
        \item $X$ is reflexive; or
        \item $X$ has nontrivial type; or
        \item $X\not\supset c_0$ and $\Cc([0,1];U)$ has property (V); or
        \item $X$ is \WCG{} and $\Ll^\infty([0,1];U)$ is a Grothendieck space;
    \end{itemize}
    then for any $u \in \Ll^\infty([0,\infty);U)$ the corresponding
    mild solution of \eqref{eq:sys} is a continuous function taking
    values in $X$.
\end{corollary}
 \begin{corollary}
    Let $B_0 \in \LL(U,X)$.
     If $\lambda \in \rho(A)$ and $B = (\lambda\id_X - A)B_0$ is even $\Ee_F$-admissible for some Young function
     $F$, then $\Phi\colon u \mapsto A\int_0^{\,\cdot\,}T(\,\cdot\, - s) B_0 \dd s$ is a bounded
     operator $\Ee_F([0,1];U) \to \Cc([0,1];X)$.
 \end{corollary}
 \begin{proof}
 Let $u \in \Ee_F([0,1];U)$.
     Since $\Phi$ is the mild solution operator corresponding to  \eqref{eq:sys}, we only need to show the norm inequality 
     $\norm{\Phi u}_{\Cc([0,1];X)} \leq M \norm{u}_{\Ee_F([0,1];U)}$
     with some absolute constant $M \leq 0$. 
We have
\[
    \norm{\Phi u}_{\Cc([0,1];X)} =
    \sup_{0\leq t \leq 1}\Norm{A\int_0^t T(t-s)B_0 u(s) \dd s}_X\leq \sup_{0\leq t \leq 1} \kappa(t)\norm{u}_{\Ee_F([0,t];U)}
\]
by admissibility. It follows that $\norm{\Phi u}_{\Cc([0,1];X)} \leq \kappa(1) \norm{u}_{\Ee_F([0,1];U)}$, ending the proof.
    \end{proof}
    Under any of the assumptions of Corollary \ref{cor:linftymildsol}, we obtain continuity
    of the mild solutions for any $\Ll^\infty$-admissible $B$, partially answering Weiss' question \cite[Prob.\ 2.4]{weiss_admissibility_1989} in the positive.
\subsection{Consequences of positivity assumptions}
If additional assumptions in regards to positivity are made, then there are classes
of observation and control operators that are automatically $\Ll^1$- or $\Cc$-admissible, see, e.g., \cite[Sects.\ 3, 4]{aroraLimitcaseAdmissibilityPositive2025a}
and \cite[Sect.\ 4]{wintermayrPositivityPerturbationTheory2019a}. See also  \cite{barbieriPerturbationsPositiveSemigroups2025} and \cite{batkaiPerturbationsPositiveSemigroups2018}. 
In particular, we 
work with the assumptions made in \cite[Assumptions 3.1, 4.1]{aroraLimitcaseAdmissibilityPositive2025a}, and we refer to \cite{aroraLimitcaseAdmissibilityPositive2025a} for the terminology used.
The automatic admissibility results from \cite[Prop.\ 3.2, Thm.\ 4.9]{aroraLimitcaseAdmissibilityPositive2025a}
have the following consequences:
\begin{corollary}
    Suppose $X$, $Y$ and $U$ are ordered Banach spaces, and assume that
    $X$ has a generating and normal cone. Let $T$ be a \czero-semigroup of positive operators on $X$.

    Under the following assumptions, every positive $C \in \LL(X_{1},Y)$ is 
    $\Ee_F$-admissible for some Young function $F$:
    \begin{itemize}
        \item the norm of $Y$ is additive on $Y_+$ (e.g., $Y$ is an AL-space); 
        \item $X$ does not contain a complemented copy of $\ell^1$ (e.g.,
        if $X$ is reflexive or has nontrivial type).
    \end{itemize}

    If the following hold, then every positive $B\in\LL(U,X_{-1})$ is 
    $\Ee_F$-admissible for some Young function $F$:
    \begin{itemize}
        \item the open unit ball of $U$ is upward directed (e.g., $U$ is an AM-space);
        \item $X$ is reflexive or has nontrivial type or $\Cc([0,1];U)$
        has Pełczyński's property (V).
    \end{itemize}
\end{corollary}
\begin{proof}
    The assertion on observation operators follows from \cite[Prop.\ 3.2]{aroraLimitcaseAdmissibilityPositive2025a}, which states that an operator 
$C$ satisfying the above assumptions must be $\Ll^1$-admissible and \ref{thm:OrliczforCs},
which implies that $C$ is even $\Ee_F$-admissible for some Young function $F$.

For control operators, \cite[Thm.\ 4.9]{aroraLimitcaseAdmissibilityPositive2025a}
implies that any $B$ with the above properties is $\Cc$-admissible. The $\Ee_F$-admissibility
then follows from Corollary \ref{cor:reflexiveforB}, Theorem \ref{thm:nikishinforB}
or Corollary \ref{cor:pelczynski_for_B}, respectively.
\end{proof}
\subsection{Maximal regularity and the theorems of Guerre-Delabriere and Baillon}
The classical maximal regularity with respect to $\Cc$ or $\Ll^1$ is considered `exotic' by the dichotomy results due to Baillon \cite{baillonCaractere1980} and Guerre-Delabriere \cite{guerre-delabriere$L_p$regularityCauchyProblem1995}. Both results show that this type of maximal regularity can only occur if either the generator is bounded or the underlying Banach space contains a (complemented) copy of $c_{0}$ ($\ell^{1})$. Note that in PDE applications, the case of bounded generators is typically irrelevant.
The existing proof of Baillon's theorem---see Baillon's original article \cite[Theorème 1]{baillonCaractere1980} and compare with \cite[Thm.\ 17.4.4]{hytonenAnalysisBanachSpaces2023}, \cite[Thm.\ 0.5]{eberhardtBaillonsTheoremMaximal1992}---relies on a characterization of spaces containing $c_{0}$. More precisely, by assuming that $A$ is unbounded, one can construct a sequence which leads to a Schauder basis of an isomorphic copy of $c_{0}$, see also \cite[Prop.\ 2.5]{jacobRefinementBaillonTheorem2022}. 
On the other hand, Guerre-Delabriere's argument rests on a duality argument and, in turn, on Baillon's theorem. 

Using our results, we give
an alternative and direct proof of Guerre-Delabriere's result \cite[Prop.\ 1]{guerre-delabriere$L_p$regularityCauchyProblem1995} without making use of Baillon's
theorem.
\begin{theorem}[Guerre-Delabriere; {\cite[Prop.\ 1]{guerre-delabriere$L_p$regularityCauchyProblem1995}}]\label{thm:guerredelabriere}
    Let $X$ be a Banach space that does not contain $\ell^1$ as a
    complemented subspace. If $A$ has maximal $\Ll^1$-regularity, then $A$
    is bounded.
\end{theorem}
\begin{proof}
    Suppose $A$ has maximal $\Ll^1$-regularity. Due to \cite[Thm.~3.6]{kaltonREMARKSL1MAXIMAL2008}, see also \cite[Prop.\ 2.1]{jacobRefinementBaillonTheorem2022} $A$ is an $\Ll^1$-admissible observation operator for the
    semigroup generated by $A$. As $X$ does not contain a
    complemented subspace isomorphic to $\ell^1$, Theorem \ref{thm:OrliczforCs} implies that $A$ is zero-class $\Ll^1$-admissible.
    By \cite[Prop.\ 16]{jacobContinuitySolutionsParabolic2019}, $A$ is bounded.
\end{proof}
We can also prove a generalization of Baillon's result for structured continuous maximal regularity; however, we need a slightly stronger assumption on the Banach space $X$.

\begin{theorem}[Partial generalization of Baillon's result]
   \label{thm:Baillongen}  Let $A$ generate a strongly continuous semigroup on a Banach space $X$ such that  $X^{\odot\prime}$ contains no copy of $c_{0}$ and let $B_{0}\in \LL(U,X)$. If $A$ has  $B_{0}$-structured $\Cc$-maximal regularity, then $\lim_{\lambda\to\infty}\lambda R(\lambda,A)B_{0}=B_{0}$ in $\LL(U,X)$. 
\end{theorem}
\begin{proof}
Suppose now that $X^{\odot\prime}$ contains no copy of $c_{0}$. Thus, by \cite[Thm.\ 4]{Bessaga1958}, $X^\odot$ contains no complemented copy of $\ell^{1}$. 
By the sun duality Proposition \ref{prop:sun_duality} above, $B_{0}'A^{\odot}$ is an $\Ll^1$-admissible observation operator on $X^{\odot}$, which is zero-class since $X^{\odot}$ does not contain $\ell^{1}$ complementedly. By \cite[Thm.\ 3.1]{aroraAdmissibleOperatorsSundual2025}, $A_{-1}B_{0}$ is zero-class $C$-admissible. Using the relation between the Laplace transform and the resolvent, we conclude that 
$R(\lambda,A)A_{-1}B_{0}$ converges to $0$ in the operator topology as $\lambda\to\infty$. The resolvent identity shows the assertion.
\end{proof}
\begin{remark}
Some notes on the relation to Baillon's theorem are in order:
\begin{enumerate}
    \item Baillon's theorem states that if $X$ does not contain $c_{0}$ and $A$ has continuous maximal regularity, i.e., $B_0$-structured $C$-maximal regularity with $B_{0}=\id_X$, then $A$ is bounded. Since $X$ is canonically embedded in ${X^{\odot\prime}}$, see, e.g., \cite[7-8]{neervenAdjointSemigroupLinear1992},
    $X^{\odot\prime}$ contains $c_0$ if $X$ does. Thus, the assumption of $X^{\odot\prime}$ not containing $c_{0}$ is stronger than the one in Baillon's result.
    However, in the case $B_{0}=\id_X$, the assertion of our Theorem \ref{thm:Baillongen}, that is, $\lim_{\lambda\to\infty}\lambda R(\lambda,A)=\id_X$, implies that $A$ is a bounded operator. This follows by choosing $\lambda$ such that $\norm{\lambda R(\lambda,A)-\id_X}<1$ and a Neumann series argument.
    \item A result of Bessaga and Pełczyński \cite[Thm.\ 4]{Bessaga1958}, see also \cite[Prop.\ 2.e.8]{lindenstraussClassicalBanachSpaces1977} implies that $X^\odot$ contains $\ell^1$
   as a complemented subspace if and only if ${X^{\odot\prime}}$ contains $c_0$. 
    The Bourgain--Delbaen space $Y$, see \cite[Ch.\ 5]{bourgainClassSpecialLaSpaces1980}, cf.\ also \cite[Thm.\ 487]{guiraoRenormingsBanachSpaces2022},
    has the property that $Y\not\supset c_0$ but $Y' \simeq \ell^1$ such that $Y''\simeq \ell^\infty \supset c_0$. 
    However, we do not know if there is a (non-trivial) \czero-semigroup on $Y$ such that
    $Y^{\odot\prime}\supset c_0$. For example, $Y$ could have the Lotz property, i.e., 
    that every \czero-semigroup on $Y$ has a bounded generator; cf., \cite{vanneervenConverseLotz1992}. Lotz showed 
    in \cite[Thm.\ 3]{lotzUniformConvergenceOperators1985} that semigroups on Grothendieck spaces
    with the Dunford--Pettis property (see, e.g., \cite[208]{lotzUniformConvergenceOperators1985}) have bounded generators; whence such spaces are Lotz.
    Regarding the Bourgain--Delbaen space $Y$: it is known that $Y$ is not a Grothendieck space
    \cite[Ex.\ 4.3.1]{gonzalez_grothendieck_2021} although $Y$ has the Dunford--Pettis property \cite[155]{bourgainClassSpecialLaSpaces1980}. It is also known in the folklore
    that $Y$ is not a Banach lattice (see, e.g., \cite[46]{talagrandDualBanachLattices1981}), which means that the characterization
    of the Lotz property on Banach lattices \cite[Thm.\ 2]{vanneervenConverseLotz1992} cannot
    be used.
    \item We also note that Baillon's theorem for Banach spaces $X$ not containing $c_{0}$ cannot be proven in full generality with any argument
    that yields weak compactness of all operators $\Cc([0,1];X) \to X$. This is a result 
    of the known fact that every operator 
    $\Cc([0,1];X) \to X$ is weakly compact only if $X$ is reflexive, see, e.g., 
    \cite[Proof\ Thm.\ 1.6]{emmanueleBanachSpacesProperty1988}. Hence, potential proofs of Baillon's
    theorem using weak compactness must be adapted to the special structure of input maps $\Phi u=A\int_{0}^{1}T(1-s)u(s)\dd s$.
\item The assumption that $X^{\odot\prime}$ does not contain $c_{0}$ is satisfied in the following special situations:
    \begin{enumerate}
        \item $X^\odot$ does not contain $\ell^1$ as complemented subspace; 
        \item $X$ is reflexive;
        \item $X$ has nontrivial type;
        \item $X^{\odot\prime}$ has finite cotype;
        \item $X$ has finite cotype and $X^{\odot\prime} = X''$.
    \end{enumerate}
     By the second item in this remark, a) is equivalent to $X^{\odot\prime}$ containing $c_{0}$. Since nontrivial type of $X$ implies that $X'$ and hence $X^{\odot}$ have nontrivial type as well, $X^{\odot}$ does not contain a copy of $\ell^{1}$. 
     Maurey's characterization of finite cotype, see \cite{hytonenAnalysisBanachSpaces2018}, directly shows that d) implies that $X^{\odot\prime}$ does not contain $c_{0}$. 
     By \cite[Cor.\ 11.9]{diestelAbsolutelySummingOperators1995}, \cite[p.\ 95, (b)]{wojtaszczykBanachSpacesAnalysts1991}, the principle
    of local reflexivity implies that $\operatorname{cotype}(X) = \operatorname{cotype}(X'')$. Thus e) implies d).  
\end{enumerate}
    \end{remark}
\appendix
\section{Appendix}
\subsection{Pełczyński's property (V)}
It seems to be well-known in the
literature that property (V) transfers to complemented subspaces, cf., e.g., \cite{Kalton1985,Randrianantoanina1996}. For completeness, we give a proof:
\begin{lemma}\label{lem:compl(V)}
    If $X$ is a Banach space with Pełczyński's property (V) and $Y$ is a complemented subspace 
    of $X$, then $Y$ has Pełczyński's property (V).\label{lem:Vforcomplementedsubspaces}
\end{lemma}
\begin{proof}
    Let $X$ have Pełczyński's property (V) and let $Y \subset X$ be a complemented subspace. Consider
    another Banach space $Z$ and a unconditionally converging operator $T\in\LL(Y,Z)$. To show
    that $Y$ has (V), we need to show that $T$ is weakly compact.

    Since $Y$ is complemented in $X$, there is a projection $P_Y\in\LL(X)$
    with $Y$ as its range. Hence $\tilde T\colon X \to Z$ , $x \mapsto TP_Y x$
        is a well-defined bounded operator. 
        
        Moreover, if $\sum x_n$ is a \wuc{} series in $X$, then, by definition, $\sum \abs{\spr{x_n}{x'}_{X,X'}}$ converges 
    for any $x' \in X'$. Then $\sum P_Y x_n$ is \wuc{} in $Y$: for any $y' \in Y'$ it follows that $y' \circ P_Y$ is an 
    element of $X'$. By the \wuc\ property of $\sum x_n$ it follows that $\sum\abs{\spr{x_n}{y'\circ P_Y}_{X,X'}} = \sum\abs{\spr{P_Yx_n}{y'}_{Y,Y'}}$ converges. Hence $\sum P_Y x_n$ is \wuc\ in $Y$.
    As $T$ is an \uc\ operator on $Y$, it follows that $\sum TP_Y x_n$ converges
    unconditionally in $Z$. Therefore the operator $\tilde T = TP_Y$ is \uc{} 

    Since $X$ has (V), $\tilde T$ is weakly compact. Hence
    $\tilde T(\ball_Y)$ is relatively weakly compact in $Z$. As $\tilde T(\ball_Y) = T(\ball_Y)$,
    $T$ is a weakly compact operator $T\colon Y \to Z$.
    Thus $Y$ has property (V).
\end{proof}
\begin{remark} \label{rem:Pelczynski_2}
    As remarked in 
    \cite[Proof Thm.\ 6]{Kalton1985} and \cite[Proof Thm.\ 1]{Randrianantoanina1996},
    a subspace isomorphic to $X$ is complemented in $\Cc([0,1];X)$; indeed, the projection 
    $f \mapsto f(0)\chi_{[0,1]}$ has range $\{ x\chi_{[0,1]} : x \in X\}\simeq X$.
    Thus, by Lemma \ref{lem:compl(V)},
    $X$ has property (V) if $\Cc([0,1];X)$ does. In particular, if $X \not\supset c_0$ and $\Cc([0,1];X)$
    has (V), by \cite[Thm.\ 8 (B)]{pelczynskiBanachSpacesWhich1962}, $X$ is reflexive.
\end{remark}
\subsection{Structured maximal regularity and admissibility}
The goal of this section is to show that the structured continuous maximal regularity
property from the introduction is indeed equivalent to $\Cc$-admissibility. The setting is as follows:
Let $A$ generate a \czero-semigroup $T$ on a Banach space $X$, and suppose without loss of generality that $0 \in \rho(A)$. Then let $U$ be another 
Banach space and $B_0 \in \LL(U,X)$ be an operator. Given $B_0$ and identifying $A$
with its extension to a generator on $X_{-1}$, we can ask 
whether $B \coloneq A B_0 \in \LL(U,X_{-1})$ is $\Cc$-admissible.

The equivalence of (classical) continuous maximal regularity and admissibility
of the generator is known in the literature; the statement follows from Travis' results in \cite[Lem.\ 3.1, Lem.\ 3.2, Prop.\ 3.1]{travisDifferentiabilityWeakSolutions1981}, where the result is 
not phrased using the terminology of admissible operators; in \cite[Prop.\ 2.2]{jacobRefinementBaillonTheorem2022} the authors note that the results in \cite{travisDifferentiabilityWeakSolutions1981} can be seen through the lens of admissibility theory, while proving that both statements are also equivalent to 
their respective versions for regulated functions.

We say that $A$ has $B_0$-structured $\Cc$-maximal regularity if 
\[
    \Phi\colon \Cc([0,1];U) \to \Cc([0,1];X),\qquad f\mapsto A\int_0^{\,\cdot\,} T(\,\cdot\, - s) B_0 f(s) \dd s 
\]
is a well-defined and bounded operator. It is clear that in the case $U = X$
and $B_0 = \id_X$ the structured maximal regularity reduces to classical
maximal regularity.

See, e.g., \cite[Def.\ 3.13]{kruseContinuousMaximalRegularity2025} where the authors discuss structured maximal regularity, in a more general context, leaving Banach spaces,
using the wording ``$(T(t))_{t\geq 0}$ satisfies $\Cc$-maximal regularity for 
$(B,r)$'', where $r>0$ is the length of the interval. In \cite[Cor.\ 5.4]{kruseContinuousMaximalRegularity2025} the authors show that admissibility of $B$
and structured maximal regularity of $B_0$ are equivalent under the assumption 
that $B_0$ commutes with the semigroup; we show that (in the Banach space case) 
the commutation condition can be removed, adapting Travis' proof to the structured
case.

The proofs in the literature use the concept of \emph{bounded semivariation} of operator-valued functions;
see, e.g., \cite[426]{travisDifferentiabilityWeakSolutions1981} for a definition. To 
show equivalence of $\Cc$-admissibility and maximal regularity, we will use the fact
that $\Cc$-admissibility of $B$ is equivalent to bounded semivariation of 
$T(\,\cdot\,)B_0$ by \cite[Prop.\ 2.1]{aroraAdmissibleOperatorsSundual2025}. 
Travis' proof also relies on Riemann--Stieltjes integrals with respect to operator-valued 
functions, see, e.g., \cite[Prop.\ 2.1]{hoenigGreenFunctionLateral1973}. See also the references cited in \cite{hoenigGreenFunctionLateral1973}, especially \cite[Sect.\ 4]{battDarstellungLinearerTransformationen1959}, where an introduction to Riemann--Stieltjes
integration with operator-valued functions can be found. Compare also \cite[Ch.\ 1.9]{arendtVector-valuedLaplaceTransforms2011}.
\begin{theorem}\label{thm:structuredmaxreg}
    Let $A$ generate a \czero-semigroup on the Banach space $X$ and let $\lambda \in \rho(A)$. If $U$ is another Banach space and $B_0 \in \LL(U,X)$, then $A$ has 
    $B_0$-structured $\Cc$-maximal regularity if and only if $B \coloneq (\lambda\id_X - A)B_0$
    is $\Cc$-admissible for $A$.
\end{theorem}
\begin{proof}
We assume that $\lambda = 0$ without loss of generality, such that we can consider $B = AB_0$. Suppose $A$ has structured $\Cc$-maximal regularity with respect to $B_0$. Let $f \in 
\Cc([0,1];U)$. By
definition, 
\[
   \sup_{t\in[0,1]} \Norm{A\int_0^{t} T(t - s) B_0f(s) \dd s}_U \leq C \norm{f}_{\Cc([0,1];U)};
\]
it directly follows that
\[
    \Norm{\int_0^{1} T(t - s) AB_0f(s) \dd s}_U \leq \sup_{t\in[0,1]} \Norm{A\int_0^{t} T(t - s) B_0f(s) \dd s}_U \leq C \norm{f}_{\Cc([0,1];U)},
\]
which implies that $B = AB_0$ is $\Cc$-admissible for $A$.

For the converse, we use the fact that, by \cite[Prop.\ 2.1]{aroraAdmissibleOperatorsSundual2025},
$\Cc$-admissibility implies that $t \mapsto T(t)B_0$ has bounded semivariation on $[0,\tau]$ for
any fixed $\tau > 0$. To complete the proof, adapting the proofs of Lemmas 3.1 and 3.2 
in \cite{travisDifferentiabilityWeakSolutions1981} to the structured case, we 
show that structured maximal regularity can be recovered from bounded semivariation.

Once we replace $T(\,\cdot\,)$ by $T(\,\cdot\,)B_0$, Travis' proof of Lemma 3.1 can be carried over exactly and we get, for $0 \leq t \leq 1$ and $f\in\Cc([0,1];U)$,
\[
    A\int_0^t T(t-s)B_0f(s) \dd s = - \int_0^t \dd_s T(t-s)B_0f(s);
\]
in particular, the left hand integral lies in $D(A)$.
We note that there seems to be a sign mistake in this step in \cite{travisDifferentiabilityWeakSolutions1981}.

To show that the Riemann--Stieltjes integral is continuous in the parameter $t \in [0,1]$,
in Lemma 3.2, Travis considers left and right continuity separately.
For right continuity, once again, the argument applies almost verbatim to $T(\,\cdot\,)B_0$
in place of the semigroup $T$.

For completeness, we sketch the proof of right continuity in the structured case, following
the proof of \cite[Lem.\ 3.2]{travisDifferentiabilityWeakSolutions1981}.
Let $t \in [0,\tau)$ and $0 < h \ll 1$ small enough. We have
\begin{align*}
    &\int_0^{t+h} \dd_s T(t + h - s) B_0f(s) - \int_0^t \dd_s T(t-s) B_0 f(s) \\
    &= \int_0^t \dd_s T(t + h - s) B_0f(s) - \int_0^t \dd_s T(t-s) B_0 f(s) + \int_t^{t+h} \dd_s T(t + h - s) B_0f(s)\\
    &= \int_0^t [T(h) - \id_X] \dd_s T(t-s)B_0 f(s) + \int_t^{t+h} \dd_s T(t + h - s) B_0f(s)\\
    &= [T(h) - \id_X] \int_0^t \dd_s T(t-s)B_0 f(s) + \int_t^{t+h} \dd_s T(t + h - s) B_0f(s)\\
    &= [T(h) - \id_X] \int_0^t \dd_s T(t-s)B_0 f(s) + \int_0^{h} \dd_s T(s) B_0f(t + h - s)\\
    &\eqqcolon I_1 + I_2;
\end{align*} 
again, there seems to be a discrepancy between this expression and the analogous one given in 
\cite{travisDifferentiabilityWeakSolutions1981} (in \cite{travisDifferentiabilityWeakSolutions1981},
the term corresponding to $I_2$ is $\int_0^h \dd_s T(t-s) f(t+h-s)$). However, the argument given by Travis
still applies: $I_1$ tends to $0$ as $h \searrow 0$ due to strong continuity of 
$T$ and $I_2 \to 0$ as $h \searrow 0$ by \cite[Lem.\ 2.1]{travisDifferentiabilityWeakSolutions1981}. Hence $t \mapsto \int_0^t \dd_s T(t-s) B_0f(s)$
is right continuous.

For left continuity, we set up the integrals in the same way as Travis:
Let $t \in (0,\tau]$ and $0 < h \ll 1$, and consider
\begin{align*}
    &\int_0^t \dd_s T(t-s) B_0 f(s) - \int_0^{t-h} \dd_s T(t-h - s)B_0f(s)\\
    &= \int_0^{t-h} \dd_s T(t-s) B_0 f(s) + \int_{t-h}^t \dd_s T(t-s) B_0f(s)
    - \int_0^{t-h} \dd_s T(t-h - s)B_0f(s)\\
    &\eqqcolon J_1 + J_2 - J_3.
\end{align*}
The term $J_2$ can be seen to converge to $0$ as $h \searrow 0$ by \cite[Lem.\ 2.1]{travisDifferentiabilityWeakSolutions1981}, which is as in the non-structured case.

For $J_1 - J_3$, as $T(\,\cdot\,)$ does not necessarily have bounded semivariation, Travis' argument 
does not carry over to the structured case without additional assumptions such 
as $B$ commuting with $T(\,\cdot\,)$ or uniform continuity of the semigroup.
We therefore use the following alternative argument:
\begin{align*}
    &J_1 - J_3 = \int_0^{t-h} \dd_s T(t-s) B_0 f(s) - \int_0^{t-h} \dd_s T(t-h - s)B_0f(s)\\
    &= \int_0^{t-h} \dd_s T(t-s) B_0 f(s) - \int_{h}^{t} \dd_s T(t- s)B_0f(s-h)\\
    &= \int_0^{h} \dd_s T(t-s) B_0 f(s) +
    \int_h^{t-h} \dd_s T(t-s)B_0[f(s) - f(s-h)]- \int_{t-h}^{t} \dd_s T(t- s)B_0f(s-h)\\
    &\eqqcolon K_1 + K_2 - K_3.
\end{align*}
As before, $K_1 \to 0$ as $h\searrow0$ by \cite[Lem.\ 2.1]{travisDifferentiabilityWeakSolutions1981}; the same result implies
that 
\[
    K_3 = \int_0^h \dd_s T(h-s)B_0 f(s+t-2h) \to 0
\]
as $h \searrow 0$. It only remains to show $K_2$, the integral over an interval
of length $t - 2h \not\to 0$, also converges to zero.

By \cite[Prop.\ 1]{hoenigGreenFunctionLateral1973}, since $T(\,\cdot\,)B_0$ has
bounded semivariation on $[0,t]$, we have 
\[
    \norm{K_2}_X = \Norm{\int_h^{t-h} \dd_s T(t-s)B_0[f(s)- f(s-h)]}_X
    \leq \operatorname{SV}(T(\,\cdot\,)B_0)\norm{f(\,\cdot\,) - f(\,\cdot\, - h)}_{
    \Cc([h,t-h];U)}.
\]
Since $f$ is continuous on the compact interval $[0,t]$,
$f$ is also uniformly continuous there. Thus 
\[
    \norm{f(\,\cdot\,) - f(\,\cdot\, - h)}_{
    \Cc([h,t-h];U)} = \sup_{s\in[h,t-h]}\norm{f(s) - f(s - h)}_U \to 0
\]
as $h \to 0$. It follows that $K_2 \to 0$.

We conclude that $t \mapsto \int_0^t \dd_s T(t-s) B_0f(s)$ is also 
left continuous on $[0,\tau]$, such that $\Phi f = A\int_0^{\,\cdot\,} T(\,\cdot\, - s) B_0f(s)\dd s \in \Cc([0,\tau];X)$ for $f\in \Cc([0,\tau];U)$. Moreover, 
\[
    \norm{\Phi f}_{\Cc([0,1];X)} = \sup_{r\in[0,t]} \Norm{A\int_0^{r} T(r - s) B_0f(s)\dd s}_X = \sup_{r\in[0,t]} \Norm{\int_0^r \dd_s T(r-s)B_0f(s)}_X
\]
so, using the bounded semivariation of $T(\,\cdot\,)B_0$, the inequality from \cite[Prop.\ 1]{hoenigGreenFunctionLateral1973} yields
\[
    \sup_{r\in[0,t]} \Norm{\int_0^r \dd_s T(r-s)B_0f(s)}_X \leq
    \sup_{r\in[0,t]}\operatorname{SV}(T(r-\,\cdot\,)B_0) \norm{f}_{\Cc([0,r];U)}
    \leq C \norm{f}_{\Cc([0,t];U)}
\]
with some $C \geq 0$. Hence $A$ has $B_0$-structured maximal regularity.
\end{proof}
\printbibliography
\end{document}